\documentclass{amsart}
\usepackage{graphicx}
\usepackage{amsmath, amsfonts,amssymb}
\usepackage{amsthm}
\usepackage{mathrsfs}
\usepackage{etex}
\usepackage[all]{xy}
\usepackage[colorlinks,
             linkcolor=black,
             citecolor=blue,
             bookmarksnumbered=true]{hyperref}
\usepackage{appendix}
\usepackage{bbold}

\newtheorem{prop}{Proposition}[section]
\newtheorem{thm}{Theorem}[section]
\newtheorem{cor}{Corollary}[section]
\newtheorem{lem}[thm]{Lemma}

\theoremstyle{definition}
\newtheorem{dfn}{Definition}[section]
\newtheorem{ex}{Example}
\newtheorem*{clm}{Claim}
\theoremstyle{remark}
\newtheorem*{rmk}{Remark}
\theoremstyle{remark}

\theoremstyle{remark}

\newcounter{alphasect}
\def\alphainsection{0}

\let\oldsection=\section
\def\section{%
  \ifnum\alphainsection=1%
    \addtocounter{alphasect}{1}
  \fi%
\oldsection}%

\renewcommand\thesection{%
  \ifnum\alphainsection=1%
    \Alph{alphasect}
  \else%
    \arabic{section}
  \fi%
}%

\newenvironment{alphasection}{%
  \ifnum\alphainsection=1%
    \errhelp={Let other blocks end at the beginning of the next block.}
    \errmessage{Nested Alpha section not allowed}
  \fi%
  \setcounter{alphasect}{0}
  \def\alphainsection{1}
}{%
  \setcounter{alphasect}{0}
  \def\alphainsection{0}
}%

\newcommand{\G}         {\mathcal{G}}

\newcommand{\h}         {\mathcal{H}}

\newcommand{\K}         {\mathcal{K}}

\newcommand{\yg} {\tilde y\left(\G\right)}

\newcommand{\ps}        {\operatorname{Set}^{\C^{op}}}
\newcommand{\pst}        {Gpd^{\C^{op}}}
\newcommand{\Prst} {\mathfrak{Pres}St_J\left(\T\right)}

\newcommand{\Zt}{\raise0.7ex\hbox{$\mathbb{Z}$} \!\mathord{\left/
 {\vphantom {\mathbb{Z} {2\mathbb{Z}}}}\right.\kern-\nulldelimiterspace}
\!\lower0.7ex\hbox{${2\mathbb{Z}}$}}

\newcommand{\holim}{\operatorname{holim}}

\newcommand{\hl}        {\underset{{\longleftarrow\!\!\!-\!\!\!-\!\!\!-\!\!\!-\!\!\!-\!\!\!-} } \holim  \:}
\newcommand{\hc}        {\underset{{-\!\!\!-\!\!\!-\!\!\!-\!\!\!-\!\!\!-\!\!\!\longrightarrow}} \holim  \:}

\newcommand{\Hom}{\operatorname{Hom}}
\newcommand{\Map}{\mathbb{M}\!\operatorname{ap}}
\newcommand{\St}{\operatorname{St}}
\newcommand{\Set}{\operatorname{Set}}

\def\X{\mathscr X}
\def\Y{\mathscr Y}
\def\Z{\mathscr Z}

\def\C{\mathscr C}

\def\Ee{\mathbb{E}}

\newcommand{\T} {\mathbb{TOP}}
\newcommand{\ch} {\mathbb{CH}}
\newcommand{\cg} {\mathbb{CGH}}
\newcommand{\cgr} {\mathbb{CG}}
\newcommand{\cgh} {\mathbb{CGH}}
\newcommand{\kp}   {\mathbb{CH}}
\newcommand{\lk}   {\mathbb{LCH}}

\newcommand{\Haus} {\mathbb{HAUS}}

\newcommand{\ju} {j^{*}}
\newcommand{\jd} {j_{*}}
\newcommand{\js} {j_{!}}
\newcommand{\Skp} {\operatorname{Set}^{{\kp}^{op}}}
\newcommand{\Sg} {\operatorname{Set}^{{\cg}^{op}}}

\newcommand{\Psg} {Gpd^{{\cg}^{op}}}
\newcommand{\Psk} {Gpd^{{\kp}^{op}}}
\newcommand{\yk} {y_{\kp}}
\newcommand{\ycg} {y_{\cg}}
\newcommand{\cgt} {\mathscr{CG}}
\newcommand{\kpt} {\mathscr{K}}

\newcommand{\Sgs} {Sh_{\cgt}\left(\cg\right)}
\newcommand{\Stcg} {St_{\cgt}\left(\cg\right)}

\newcommand{\Stk} {St\left(\kp\right)}
\newcommand{\Stc} {St\left(\cg\right)}

\newcommand{\Tst} {\mathfrak{TSt}}
\newcommand{\Ctst} {\mathscr{CG}\mathfrak{TSt}}

\newcommand{\QCtst} {\mathfrak{QuasiTSt}}

\newcommand{\rt} { \rotatebox[origin=C]{90}{$\perp$} }
\newcommand{\rightt} {\: \rt \:}

\def\C{\mathscr C}
\def\X{\mathscr X}
\def\Y{\mathscr Y}
\def\Z{\mathscr Z}

\DeclareMathAlphabet{\scr}{OT1}{pzc}%
                                 {m}{it}

\def\rrrarrow{\hspace{.05cm}\mbox{\,\put(0,-3){$\rightarrow$}\put(0,1){$\rightarrow$}\put(0,5){$\rightarrow$}\hspace{.45cm}}}
\def\rrarrow{  \hspace{.05cm}\mbox{\,\put(0,-2){$\rightarrow$}\put(0,2){$\rightarrow$}\hspace{.45cm}}}
\def\acts{\hspace{.1cm}{\setlength{\unitlength}{.30mm}\linethickness{.09mm}
                        \begin{picture}(8,8)(0,0)\qbezier(7,6)(4.5,8.3)(2,7)\qbezier(2,7)(-1.5,4)(2,1)\qbezier(2,1)(4.5,-.3)(7,2)
                                                 \qbezier(7,6)(6.1,7.5)(6.8,9)\qbezier(7,6)(5,6.1)(4.2,4.4)
                        \end{picture}\hspace{.1cm}}}

\def\acted{\hspace{.1cm}{\setlength{\unitlength}{.30mm}\linethickness{.09mm}
                        \begin{picture}(8,8)(0,0)\qbezier(1,6)(3.5,8.3)(6,7)\qbezier(6,7)(9.5,4)(6,1)\qbezier(6,1)(3.5,-.3)(1,2)
                                                 \qbezier(1,6)(1.9,7.5)(1.2,9)\qbezier(1,6)(3,6.1)(3.8,4.4)
                        \end{picture}\hspace{.1cm}}}

\newlength{\ignora}

\def\dontshow#1{}

\begin{document}

\title{Compactly Generated Stacks: A Cartesian Closed Theory of Topological Stacks}
\author{David Carchedi}

\begin{abstract}
A convenient bicategory of topological stacks is constructed which is both complete and Cartesian closed. This bicategory, called the bicategory of compactly generated stacks, is the analogue of classical topological stacks, but for a different Grothendieck topology. In fact, there is an equivalence of bicategories between compactly generated stacks and those classical topological stacks which admit locally compact Hausdorff atlases. Compactly generated stacks are also equivalent to a bicategory of topological groupoids and principal bundles, just as in the classical case. If a classical topological stack and a compactly generated stack have a presentation by the same topological groupoid, then they restrict to the same stack over locally compact Hausdorff spaces and are homotopy equivalent.
\end{abstract}

\maketitle
\markboth{David Carchedi}{Compactly Generated Stacks}


\tableofcontents

\section{Introduction}

The aim of this paper is to introduce the bicategory of compactly generated stacks. Compactly generated stacks are ``essentially the same'' as topological stacks, however, their associated bicategory is Cartesian closed and complete, whereas the bicategory of topological stacks appears to enjoy neither of these properties. In this paper, we show that these categorical shortcomings can be overcame by refining the open cover Grothendieck topology to take into account compact generation.


It is well known that the category of topological spaces is not well behaved. In particular, it is not Cartesian closed. Recall that if a category $\C$ is Cartesian closed then for any two objects $X$ and $Y$ of $\C$, there exists a mapping object $\Map\left(X,Y\right),$ such that for every object $Z$ of $\C$, there is a natural isomorphism $$\Hom\left(X,\Map\left(X,Y\right)\right) \cong \Hom\left(Z\times X,Y\right).$$ The category of topological spaces is not Cartesian closed; one can topologize the set of maps from $X$ to $Y$ with the compact-open topology, but this space will not always satisfy the above universal property. In a 1967 paper \cite{Steen}, Norman Steenrod set forth compactly generated Hausdorff spaces as a convenient category of topological spaces in which to work. In particular, compactly generated spaces are Cartesian closed. Though technical in nature, history showed this paper to be of great importance; it is now standard practice to work within the framework of compactly generated spaces.

Unfortunately, topological stacks are not as nicely behaved as topological spaces, even when considering only those associated to compactly generated Hausdorff topological groupoids. The bicategory of topological stacks is deficient in two ways as it appears to be neither complete, nor Cartesian closed \cite{NoohiH},\cite{NoohiM}; that is, mapping stacks need not exist. Analogously to the definition of mapping spaces, if $\X$ and $\Y$ are two topological stacks, a mapping stack $\Map\left(\X,\Y\right)$ (if it exists), would be a topological stack such that there is a natural equivalence of groupoids $$\Hom\left(\Z,\Map\left(\X,\Y\right)\right) \simeq \Hom\left(\Z \times \X,\Y\right),$$ for every topological stack $\Z$. The mapping stack $\Map\left(\X,\Y\right)$ always exists as an abstract stack, but it may not be a \emph{topological} stack, so we may not be able to apply all the tools of topology to it. This problem can be fixed however, as there exists a more well behaved bicategory of topological stacks, which we call ``compactly generated stacks'', which is Cartesian closed and complete as a bicategory. This bicategory provides the topologist with a convenient bicategory of topological stacks in which to work. The aim of this paper is to introduce this theory.

The study of mapping stacks has been done in many different settings. In the algebraic world, Masao Aoki and Martin Olsson studied the existence of mapping stacks between algebraic stacks in \cite{Aoki}, and \cite{Olsson} respectively. The special case of differentiable maps between orbifolds has been studied by Weimin Chen in \cite{Chen}, and is restricted to the case where the domain orbifold is compact. Andr\'e Haefliger has studied the case of smooth maps between \'etale Lie groupoids (which correspond to differentiable stacks with an \'etale atlas) in \cite{AndreH}. In \cite{Loop}, Ernesto Lupercio and Bernardo Uribe showed that the free loop stack (the stack of maps from $S^1$ to the stack) of an arbitrary topological stack is again a topological stack. In \cite{NoohiM}, Behrang Noohi addressed the general case of maps between topological stacks. He showed that under a certain compactness condition on the domain stack, the stack of maps between two topological stacks is a topological stack, and if this compactness condition is replaced with a local compactness condition, the mapping stack is ``not very far'' from being topological.


In order to obtain a Cartesian closed bicategory of topological stacks, we first restrict to stacks over a Cartesian closed subcategory of the category $\T$ of all topological spaces. For instance, all of the results of \cite{NoohiM} about mapping stacks are about stacks over the category of compactly generated spaces with respect to the open cover Grothendieck topology. We choose to work over the category of compactly generated Hausdorff spaces ($\cgh$) since, in addition to being Cartesian closed, every compact Hausdorff space is locally compact Hausdorff, which is crucial in defining the compactly generated Grothendieck topology.

There are several equivalent ways of describing compactly generated stacks. The description that substantiates most clearly the name ``compactly generated'' is the description in terms of topological groupoids and principal bundles. Recall that the bicategory of topological stacks is equivalent to the bicategory of topological groupoids and principal bundles. Classically, if $X$ is a topological space and $\G$ is a topological groupoid, the map $\pi$ of a (left) principal $\G$-bundle $P$ over $X$

$$\xymatrix @R=2pc @C=0.15pc {\G_1  \ar@<+.7ex>[d] \ar@<-.7ex>[d] & \acts & P \ar_{\mu}[lld] \ar^{\pi}[d] \\
\G_0 && X}$$
must admit local sections. If instead $\pi$ only admits local sections \emph{over each compact subset of $X,$} then one arrives at the definition of a compactly generated principal bundle. With this notion of compactly generated principal bundles, one can define a bicategory of topological groupoids in an obvious way. This bicategory is equivalent to compactly generated stacks.

There is another simple way of defining compactly generated stacks. Given any stack $\X$ over the category of compactly generated Hausdorff spaces, it can be restricted to the category of compact Hausdorff spaces $\ch$. This produces a 2-functor $$\ju:\Tst \to St \left(\ch\right)$$ from the bicategory of topological stacks to the bicategory of stacks over compact Hausdorff spaces. Compactly generated stacks are (equivalent to) the essential image of this 2-functor.

Finally, the simplest description of compactly generated stacks is that compactly generated stacks are classical topological stacks (over compactly generated Hausdorff spaces) which admit a locally compact Hausdorff atlas. In this description, the mapping stack of two spaces is usually not a space, but a stack!

For technical reasons, neither of the three previous concepts of compactly generated stacks are put forth as the definition. Instead, a Grothendieck topology $\cgt$ is introduced on the category $\cg$ of compactly generated Hausdorff spaces which takes into account the compact generation of this category. It is in fact the Grothendieck topology induced by geometrically embedding the topos $Sh \left(\ch\right)$ of sheaves over $\ch$ into the topos of presheaves $\Sg$. Compactly generated stacks are defined to be presentable stacks (see Definition \ref{dfn:presentable}) with respect to this Grothendieck topology. The equivalence of all four notions of compactly generated stacks is shown in Section \ref{sec:tcgs}.

\subsection{Why are Compactly Generated Hausdorff Spaces Cartesian Closed?}\label{subsec:cgh}
In order to obtain a Cartesian closed bicategory of topological stacks, we start with a Cartesian closed category of topological spaces. We choose to work with the aforementioned category $\cg$ of compactly generated Hausdorff spaces (also known as Kelley spaces).

\begin{dfn}\index{compactly generated space}
A topological space $X$ is \textbf{compactly generated} if it has the final topology with respect to all maps into $X$ with compact Hausdorff domain. When $X$ is Hausdorff, this is equivalent to saying that a subset $A$ of $X$ is open if and only if its intersection which every compact subset of $X$ is open.
\end{dfn}

The inclusion $v:\cg \hookrightarrow \Haus$ of the category of compactly generated Hausdorff spaces into the category of Hausdorff spaces admits a right adjoint $k$, called the Kelley functor, which replaces the topology of a space $X$ with the final topology with respect to all maps into $X$ with compact Hausdorff domain. Limits in $\cgh$ are computed by first computing the limit in $\Haus$, and then applying the Kelley functor. (In this way the compactly generated product topology differs from the ordinary product topology). In particular, $\cgh$ is a complete category.

Although the fact that this category is Cartesian closed is a classical result (see: \cite{Steen}), we will recall briefly the key reasons why this is true in order to gain insight into how one could construct a Cartesian closed theory of topological stacks.

\begin{enumerate}
	\item In $\mathbb{TOP}$, if $K$ is compact Hausdorff, then for any space $X$, the space of maps endowed with the compact-open topology serves as an exponential object $\Map\left(K,X\right)$.
	\item A Hausdorff space $Y$ is compactly generated if and only if it is the colimit of all its compact subsets:
	$$Y = \underset{K_\alpha \hookrightarrow Y} \varinjlim K_\alpha.$$
\end{enumerate}

\begin{itemize}
	\item[i)] For a fixed $Y$, for all $X$, $\Map\left(K_\alpha,X\right)$ exists for each compact subset $K_\alpha$ of $Y$.
	\item[ii)] $\cgh$ has all limits
\end{itemize}

So by general properties of limits and colimits, the space

$$\Map\left(Y,X\right):=\underset{K_\alpha \hookrightarrow Y} \varprojlim \Map\left(K_\alpha,X\right)\footnote{Technically, some of the spaces $\Map\left(K_\alpha,X\right)$ may be fail to be compactly generated, so we must actually take the limit after applying the Kelley functor to each of these mapping spaces.}$$
is a well defined exponential object (with the correct universal property).

The story starts the same for topological stacks:\\

Let $Y$ be as above and let $\X$ be a topological stack. Then $\Map\left(K_\alpha,\X\right)$ is a topological stack for each compact subset $K_\alpha \subset Y$ (see \cite{NoohiM}).

One might therefore be tempted to claim:

$$\Map\left(Y,\X\right):= \underset{K_\alpha \hookrightarrow Y} \hl \Map\left(K_\alpha,\X\right),$$
but there are some problems with this. First of all, this weak limit may not exist as a topological stack, since topological stacks are only known to have finite weak limits. There is also a more technical problem related to the fact that the Yoneda embedding does not preserve colimits (see Section \ref{sec:cgtop} for details). The main task of this paper is to show that both of these difficulties can be surmounted by using a more suitable choice of Grothendieck topology on the category of compactly generated Hausdorff spaces. The resulting bicategory of presentable stacks with respect to this topology will be the bicategory of compactly generated stacks and turn out to be both Cartesian closed and complete.

\subsection{Organization and Main Results}

Section \ref{sec:rec} is a review of some recent developments in topological groupoids and topological stacks, including some results of David Gepner and Andr\'e Henriques in \cite{Andre} which are crucial for the proof of the completeness and Cartesian closedness of compactly generated stacks. In this section, we also extend Behrang Noohi's results to show that the mapping stack of two topological stacks is topological if the domain stack admits a compact Hausdorff atlas and ``nearly topological'' if the domain stack admits a locally compact Hausdorff atlas.

Section \ref{sec:cgtop} details the construction of the compactly generated Grothendieck topology $\cgt$ on the category of compactly generated Hausdorff spaces $\cgh$. This is the Grothendieck topology whose associated presentable stacks are precisely compactly generated stacks. Many properties of the associated categories of sheaves and stacks are derived.

Section \ref{sec:cgst} is dedicated to compactly generated stacks. In Section \ref{sec:tcgs}, it is shown that compactly generated stacks are equivalent to two bicategories of topological groupoids. Also, it is shown that these are in turn equivalent to the restriction of topological stacks to compact Hausdorff spaces. Finally, it is shown that compactly generated stacks are equivalent to ordinary topological stacks (over compactly generated Hausdorff spaces) which admit a locally compact Hausdorff atlas.

Section \ref{sec:tcgs} also contains one of the main results of the paper:

\begin{thm}
The bicategory of compactly generated stacks is closed under arbitrary small weak limits.
\end{thm}

(See Corollary \ref{cor:alllimits}).

Section \ref{sec:map} is dedicated to the proof of \emph{the} main result of the paper:

\begin{thm}
If $\X$ and $\Y$ are arbitrary compactly generated stacks, then $\Map\left(\Y,\X\right)$ is a compactly generated stack.
\end{thm}

(See Theorem \ref{thm:carclos}).

This of course proves that classical topological stacks (over compactly generated Hausdorff spaces) which admit a locally compact Hausdorff atlas form a Cartesian closed and complete bicategory. We also give a concrete description of a topological groupoid presentation for the mapping stack of two compactly generated stacks.

Section \ref{sec:homotopytype} uses techniques developed in \cite{NoohiH} to assign to each compactly generated stack a weak homotopy type.

Finally, in section \ref{sec:comp}, there is a series of results showing how compactly generated stacks are ``essentially the same'' as topological stacks. In particular we extend the construction of a weak homotopy type to a wider class of stacks which include all topological stacks and all compactly generated stacks, so that their corresponding homotopy types can be compared.

For instance, the following theorems are proven:

\begin{thm}
For every topological stack $\X,$ there is a canonical compactly generated stack $\bar \X$ and a map $$\X \to \bar \X$$ which induces an equivalence of groupoids
$$\X\left(Y\right) \to \bar\X\left(Y\right)$$
for all locally compact Hausdorff spaces $Y$. Moreover, this map is a weak homotopy equivalence.
\end{thm}

(See Corollary \ref{prop:comp1}).

\begin{thm}
Let $\X$ and $\Y$ be topological stacks such that $\Y$ admits a locally compact Hausdorff atlas. Then $\Map\left(\Y,\X\right)$ is ``nearly topological'', $\Map\left(\bar\Y,\bar\X\right)$ is a compactly generated stack, and there is a canonical weak homotopy equivalence
$$\Map\left(\Y,\X \right) \to \Map\left(\bar \Y,\bar\X\right).$$
Moreover, $\Map\left(\Y,\X \right)$ and $\Map\left(\bar \Y,\bar\X\right)$ restrict to the same stack over locally compact Hausdorff spaces.
\end{thm}

(See Theorem \ref{thm:hommap}).

We end this paper by showing in what way compactly generated stacks are to topological stacks what compactly generated spaces are to topological spaces:

Recall that there is an adjunction

$$\xymatrix@C=1.5cm{\cgr \ar@<-0.5ex>[r]_-{v} & \T \ar@<-0.5ex>[l]_-{k},}$$
exhibiting compactly generated spaces as a co-reflective subcategory of the category of topological spaces, and for any space $X,$ the co-reflector $$vk\left(X\right) \to X$$ is a weak homotopy equivalence.

We prove the $2$-categorical analogue of this statement:

\begin{thm}
There is a $2$-adjunction
$$\xymatrix@C=1.5cm{\Ctst \ar@<-0.5ex>[r]_-{v} & \mathfrak{\Tst} \ar@<-0.5ex>[l]_-{k},}$$
$$v \rt \mspace{2mu} k,$$
exhibiting compactly generated stacks as a co-reflective sub-bicategory of topological stacks, and for any topological stack $\X$, the co-reflector  $$vk\left(\X\right) \to \X$$ is a weak homotopy equivalence. A topological stack is in the essential image of the $2$-functor $$v:\Ctst \to \Tst$$ if and only if it admits a locally compact Hausdorff atlas.
\end{thm}

(See Theorem \ref{thm:kellyref}.)

\vspace{0.2in}\noindent{\bf Acknowledgment:} This paper was written during my studies as a Ph.D. student at Utrecht University. I would like to thank my adviser Ieke Moerdijk for his guidance and patience; it was due to him that I came to consider the question of the presentability of the mapping stack of two topological stacks, starting first through the language of topological groupoids and \'etendues. I would also like to thank Marius Crainic for sparking my interest, in the first place, in topological groupoids and Lie groupoids in the form of a very interesting seminar as part of an MRI masterclass. I would very much like to thank Andr\'e Henriques as well for countless hours of useful mathematical discussions and suggestions. I would like to thank Andr\'e Haefliger for taking the time to explain to me his work on maps between \'etale Lie groupoids during a visit to Gen\`eve. I would also like to thank Thomas Goodwillie for help with a homotopy-theoretic proof, and Behrang Noohi for his helpful e-mail correspondence. Finally, I would like to thank the referees of this paper whose comments have been quite helpful.

 \newpage

\section{Topological Stacks}\label{sec:rec}
A review of the basics of topological groupoids and topological stacks including many notational conventions used in this section can be found in Appendix \ref{app:topst}.

\begin{rmk}
In this paper, we will denote the bicategory of topological stacks by $\Tst.$
\end{rmk}

\subsection{Fibrant Topological Groupoids}
The notion of fibrant topological groupoids was introduced in \cite{Andre}. Roughly speaking, fibrant topological groupoids are topological groupoids which ``in the eyes of paracompact Hausdorff spaces are stacks.'' The fact that every topological groupoid is Morita equivalent to a fibrant one is essential to the existence of arbitrary weak limits of compactly generated stacks. Since this concept is relatively new, in this subsection, we summarize the basic facts about fibrant topological groupoids. All details may be found in \cite{Andre}.

\begin{dfn}
The \textbf{classifying space} of a topological groupoid $\G$ is the fat geometric realization of its simplicial nerve (regarded as a simplicial space) and is denoted by $\left\|\G\right\|$.
\end{dfn}

For any topological groupoid $\G$, the classifying space of its translation groupoid $\left\|\Ee\G\right\|$ (see Definition \ref{dfn:tran})  admits the structure of a principal $\G$-bundle over the classifying space $\left\|\G\right\|$.

\begin{dfn}\cite{Andre}
Let $\G$ be a topological groupoid. A principal $\G$-bundle $E$ over a space $B$ is \textbf{universal} if every principal $\G$-bundle $P$ over a paracompact Hausdorff base $X$ admits a $\G$-bundle map

$$\xymatrix{P\ar[r]\ar[d]&E\ar@<-.5ex>[d]\\
X\ar[r]&B,}$$
unique up to homotopy.
\end{dfn}

\begin{lem}\cite{Andre}
The principal $\G$-bundle

$$\xymatrix @R=2pc @C=0.15pc {\G_1  \ar@<+.7ex>[d] \ar@<-.7ex>[d] & \acts & \left\|\Ee\G\right\| \ar_{\mu}[lld] \ar^{\pi}[d] \\
\G_0 && \left\|\G\right\|}$$
is universal.
\end{lem}

\begin{dfn}\cite{Andre}\index{fibrant topological groupoid}
A topological groupoid $\G$ is \textbf{fibrant} if the unit principal $\G$-bundle is universal. (See Definition \ref{dfn:unitbundle}.)
\end{dfn}

\begin{dfn}\cite{Andre}
The \textbf{fibrant replacement} of a topological groupoid $\G$ is the gauge groupoid of the universal principal $\G$-bundle $\left\|\Ee\G\right\|$, denoted $Fib\left(\G\right)$. (See Definition \ref{dfn:gaugegpd}.)
\end{dfn}

\begin{rmk}
If $\G$ is compactly generated Hausdorff, then so is $Fib\left(\G\right)$.
\end{rmk}

\begin{lem}\cite{Andre}
$Fib\left(\G\right)$ is fibrant.
\end{lem}

\begin{lem}\cite{Andre}
There is a canonical groupoid homomorphism $$\xi_\G: \G \to Fib\left(\G\right)$$ which is a Morita equivalence for all topological groupoids $\G$.
\label{lem:fibm}
\end{lem}

The following theorem will be of importance later:

\begin{thm}\cite{Andre}
Let $\G$ be a fibrant topological groupoid. Then for any topological groupoid $\h$ with paracompact Hausdorff object space $\h_0$, there is an equivalence of groupoids

$$\Hom_{\Tst}\left(\left[\h\right],\left[\G\right]\right) \simeq \Hom_{\T Gpd}\left(\h,\G\right)$$
natural in $\h$. In particular, the restriction of $\yg$ to the full subcategory of paracompact Hausdorff spaces agrees with the restriction of $\left[\G\right]$. (See Appendix \ref{app:topst} for the notation.)
\label{thm:imp}
\end{thm}

\subsection{Paratopological Stacks and Mapping Stacks}\label{subsection:homotopical}

Stacks on $\T$ (with respect to the open cover topology) come in many different flavors. Of particular importance of course are topological stacks, which are those stacks coming from topological groupoids. However, this class of stacks seems to be too restrictive since many natural stacks, for instance the stack of maps between two topological stacks, appear to not be topological.

A topological stack is a stack $\X$ which admits a representable epimorphism $X \to \X$ from a topological space $X$. This implies:

\begin{itemize}
\item[i)] Any map $T \to \X$ from a topological space is representable (equivalently, the diagonal $\Delta:\X \to \X \times \X$ is representable) \cite{NoohiF}.
\item[ii)] If $T \to \X$ is any map, then the induced map $T \times_{\X} X \to T$ admits local sections (i.e. is an epimorphism in $\T$).
\end{itemize}

If the second condition is slightly weakened, the result is a stack which is ``nearly topological''.

\begin{dfn}\cite{NoohiH}\label{dfn:para}\index{paratopological stack}
A \textbf{paratopological} stack is a stack $\X$ on $\T$ (with respect to the open cover topology), satisfying condition i) above, and satisfying condition ii) for all maps $T \to \X$ from a paracompact Hausdorff space\footnote{In \cite{NoohiH}, this Hausdorff condition does not explicitly appear, but it is necessary for the theorems proven in \cite{NoohiH} and \cite{NoohiM}.}.
\end{dfn}

Paratopological stacks are very nearly topological stacks:

\begin{prop}\label{prop:n2} \cite{NoohiH}
A stack $\X$ with representable diagonal is paratopological if and only if there exists a topological stack $\bar \X$ and a morphism $q:\bar \X \to \X$ such that for any paracompact Hausdorff space $T,$ $q$ induces an equivalence of groupoids

\begin{equation}
q\left(T\right):\bar \X\left(T\right) \to \X\left(T\right).
\label{eq:peq}
\end{equation}
\end{prop}

The idea of the proof can be found in \cite{NoohiH}, but is enlightening, so we include it for completeness:

If $q$ is as in (\ref{eq:peq}), and $p:X \to \bar \X$ is an atlas for $\bar \X$, then $$q \circ p:X \to \X$$ satisfies condition ii) of Definition \ref{dfn:para}, hence $\X$ is paratopological. Conversely, if $\X$ is paratopological, take $$p:X \to \X$$ as in condition ii) of Definition \ref{dfn:para}. Form the weak fibered product

$$\xymatrix{X \times_{\X} X \ar[r] \ar[d] & X \ar[d]^-{p}\\
X \ar[r]^-{p} & \X.}$$
Let $\bar \X$ be the topological stack associated with the topological groupoid $$X \times_{\X} X \rightrightarrows X,$$ and $q:\bar \X \to \X$ the canonical map.

\begin{dfn}\index{mapping stack}
Any two stacks $\X$ and $\Y$ on $\T$ have an exponential stack $\X^{\Y}$ such that

$$\X^{\Y}\left(T\right)=\Hom\left(\Y \times T,\X\right).$$
We will from here on in denote $\X^{\Y}$ by $\Map\left(\Y,\X\right)$ and refer to it as the \textbf{mapping stack} from $\Y$ to $\X$.
\end{dfn}

For the rest of this section, we work in the category $\cg$ of compactly generated Hausdorff spaces, which is Cartesian closed.

In \cite{NoohiM} Noohi proved:

\begin{thm}
If $\X$ and $\Y$ are topological stacks with $\Y \simeq \left[\h\right]$ with $\h_0$ and $\h_1$ compact Hausdorff\footnote{Noohi claimed that this works without assuming the Hausdorff condition and working with compactly generated spaces (the monocoreflective hull of $\ch$ in $\T$), however, his proof seems to have a gap without this Hausdorff assumption; within it, it is (implicitly) assumed that the product of a (non-Hausdorff) compact space with a compactly generated space is compactly generated, but a non-Hausdorff compact space is not necessarily locally compact (in the sense that each point has a neighborhood base of compact sets).}, then $\Map\left(\Y,\X\right)$ is a topological stack.
\label{thm:comap}
\end{thm}

It appears that when $\Y$ does not satisfy this rather rigid compactness condition, that this may fail (however, we will shortly release the condition for the arrow space). Noohi also proved that when $\Y$ is instead locally compact, then $\Map\left(\Y,\X\right)$ is at least paratopological:

\begin{thm}\cite{NoohiM}
If $\X$ and $\Y$ are paratopological stacks with $\Y \simeq \left[\h\right]$ such that $\h_0$ and $\h_1$ are locally compact, then $\Map\left(\Y,\X\right)$ is a paratopological stack.
\label{thm:paramap}
\end{thm}

We end this section by extending Noohi's results to work without imposing conditions on the arrow space. Firstly, Noohi's proof of Theorem \ref{thm:comap} easily extends to the case when $\h_1$ need not be compact:

First, we will need a lemma from \cite{NoohiM}:

\begin{lem}\label{lem:mapstrep}
Let $\X$ and $\Y$ be topological stacks. Then the diagonal of the stack $\Map\left(\Y,\X\right)$ is representable, i.e., every morphism $T \to \Map\left(\Y,\X\right),$ with $T$ a topological space, is representable.
\end{lem}

\begin{thm}\label{thm:topcompmapme}
If $\X$ and $\Y$ are topological stacks and $\Y$ admits a compact Hausdorff atlas, then $\Map\left(\Y,\X\right)$ is a topological stack.
\end{thm}
\begin{proof}
Fix a compact Hausdorff atlas $K \to \Y$ for $\Y$. Let $\K$ denote the corresponding topological groupoid $$K \times_{\Y} K \rightrightarrows K.$$ Let $\mathcal{U}$ denote the set of finite open covers of $K$. For each $U \in \mathcal{U},$ consider the \v{C}ech groupoid $\K_{U}$ (See Definition \ref{dfn:cech}) and $\G^{\K_{U}},$ the internal exponent of groupoid objects in compactly generated Hausdorff spaces.

Set $R_U:=\left(G^{\K_{U}}\right)_0.$ By adjunction, the canonical map $$R_U \to \G^{\K_{U}}$$ induced by the unit, produces a homomorphism $$R_U \times \K_{U} \to \G.$$ Suppose $U$ is given by $U=\left(N_i\right),$ and let $V$ be the open cover of $R_U \times K,$ given by $$\left(R_U \times N_i\right).$$ Then this homomorphism is a map $$\left(R_U \times \K\right)_V \to \G,$$ and in particular, a generalized homomorphism from $R_U \times \K$ to $\G$ (See Definition \ref{dfn:genhom}). This corresponds to a map of stacks $$R_U \times \Y \to \X,$$ which by adjunction is a morphism $$p_U:R_U \to \Map\left(\Y,\X\right).$$ Let $$R:=\underset{U \in J} \coprod R_U.$$ Then we can conglomerate these morphisms to a morphism $$p:R \to \Map\left(\Y,\X\right).$$ We will show that $p$ is an epimorphism. By Lemma \ref{lem:mapstrep}, it is representable.

Suppose that $f:T \to \Map\left(\Y,\X\right)$ is any morphism from a space $T$. We will show that $f$ locally factors through $p$ up to isomorphism. By adjunction, $f$ corresponds to a map $$\bar f:T \times \Y \to \X,$$ which corresponds to a generalized homomorphism $$\tilde f:\left(T \times \K\right)_W \to \G$$ for some open cover $W$ of $T \times K.$ Let $t \in T$ be an arbitrary point. Since $K$ is compact Hausdorff, $T \times K,$ with the classical product topology, is already compactly generated, so we can assume, without loss of generality, that each element $W_j$ of the cover $W$ is of the form $$V_j \times U_j$$ for open subsets $V_j$ and $U_j$ of $T$ and $K$ respectively. Let $t \in T$ be an arbitrary point. Then as $W$ covers the slice $\left\{t\right\} \times K,$ which is compact, there exists a finite collection $\left(U_j \times V_j\right)_{j \in A_t}$ which covers it. Let $$\mathcal{O}_t:=\underset{j \in A_t} \cap U_j.$$ Then $\mathcal{O}_t$ is a neighborhood of $t$ in $T$ such that for all $j \in A_t,$ $$\mathcal{O}_t \times V_j \subset W_j.$$ Let $U_t=\left(V_j\right)_{j \in A_t}$ be the corresponding finite cover of $K,$ and $$N_t=\left(\mathcal{O}_t \times V_j\right)_{j \in A_t}$$ the cover of $\mathcal{O}_t \times K.$ Denote the composite
$$\mathcal{O}_t \times \K_{U_t} \cong \left(\mathcal{O}_t \times \K\right)_{N_t} \to \left(T \times \K\right)_W \stackrel{\tilde f}{\longrightarrow} \G$$ by $\tilde f_t.$ By adjunction, this corresponds to a homomorphism $$\mathcal{O}_t \to G^{\K_{U_t}},$$ so the induced map of stacks $$\mathcal{O}_t \to \Map\left(\Y,\X\right)$$ factors through $p$. This induced map of stacks is the same as the map adjoint to the one induced by $\tilde f_t,$
$$\mathcal{O}_t \times \Y \to T \times \Y \stackrel{\bar f}{\longrightarrow} \X,$$ so we are done.
\end{proof}

Now, we will recall some basic notions from topology:

\begin{dfn}
A \textbf{shrinking} of an open cover $\left(U_\alpha\right)_{\alpha \in A}$ is another open cover $\left(V_\alpha\right)_{\alpha \in A}$ \emph{indexed by the same set} such that for each $\alpha,$ the closure of $V_\alpha$ is contained in $U_\alpha.$ A topological space $X$ is a \textbf{shrinking space} if and only if every open cover of $X$ admits a shrinking.
\end{dfn}

The following proposition is standard:

\begin{prop}
Every paracompact Hausdorff space is a shrinking space.
\end{prop}

\begin{thm}\label{thm:cw}
If $\X$ and $\Y$ are topological stacks such that $\Y$ admits a locally compact Hausdorff atlas, then $\Map\left(\Y,\X\right)$ is paratopological.
\end{thm}

\begin{proof}
It is proven in \cite{NoohiM} that if $\X$ and $\Y$ are topological stacks, then $\Map\left(\Y,\X\right)$ has a representable diagonal. Therefore, by Proposition \ref{prop:n2}, it suffices to prove that there exists a topological stack $\Z$ and a map $$q:\Z \to \Map\left(\Y,\X\right)$$ which induces an equivalence of groupoids $$q\left(T\right):\Z\left(T\right) \to \Map\left(\Y,\X\right)\left(T\right)$$ along every paracompact Hausdorff space $T$.

Let $\G$ be a topological groupoid presenting $\X$. Let $Y \to \Y$ be a locally compact Hausdorff atlas for $\Y$. Then we may find a covering of $Y$ by compact neighborhoods, $\left(Y_\alpha\right).$ It follows that $$\underset{\alpha} \coprod Y_\alpha \to Y \to \Y$$ is an atlas. Hence we can choose a topological groupoid $\h$ presenting $\Y$ such that $\h_0$ is a disjoint union of compact Hausdorff spaces. Let $Fib\left(\G\right)^\h$ denote the internal exponent of groupoid objects in compactly generated Hausdorff spaces. Let $$\K:=\left[Fib\left(\G\right)^\h\right].$$ Note that there is a canonical map $$\varphi:\K \to \Map\left(\Y,\X\right)$$ which sends any generalized homomorphism $T_{\mathcal{U}} \to Fib\left(\G\right)^\h$ to the induced generalized homomorphism  from $T \times \h$, $T_{\mathcal{U}} \times \h \to Fib\left(\G\right)$ (which may be viewed as object in $\Hom_{\Tst}\left(T \times \Y,\X\right)$ since $\G$ and $Fib\left(\G\right)$ are Morita equivalent).

Suppose that $T$ is a paracompact Hausdorff space. Then:

\begin{eqnarray*}
\K\left(T\right)&\simeq& \underset{\mathcal{U}} \hc \Hom_{\cg Gpd} \left(T_{\mathcal{U}},Fib\left(\G\right)^\h\right)\\
&\simeq& \underset{\mathcal{U}} \hc \Hom_{\cg Gpd} \left(T_{\mathcal{U}}\times \h,Fib\left(\G\right)\right).\\
\end{eqnarray*}
Note that any paracompact Hausdorff space is a shrinking space so without loss of generality we may assume that each cover $\mathcal{U}$ of $T$ is a topological covering by closed neighborhoods. Since any closed subset of a paracompact space is paracompact, this means that the groupoid $T_{\mathcal{U}}$ has paracompact Hausdorff object space. Moreover, the object space of $T_{\mathcal{U}}\times \h$ is the product of the compactly generated space $T_{\mathcal{U}}$ and the locally compact Hausdorff space $\h_0$, hence compactly generated. Since $T_{\mathcal{U}}$ has paracompact Hausdorff object space, and $\h_0$ is a disjoint union of compact Hausdorff spaces, the product is in fact paracompact Hausdorff by \cite{paraprod}. Finally, by Theorem \ref{thm:imp}, we have that

$$\Hom_{\cgr Gpd} \left(T_{\mathcal{U}}\times \h,Fib\left(\G\right)\right)\simeq \Hom_{\Tst}\left(T \times \Y,\X\right).$$
Hence $\K\left(T\right) \simeq \Map\left(\Y,\X\right)\left(T\right).$
\end{proof}

\section{The Compactly Generated Grothendieck Topology}\label{sec:cgtop}
\sectionmark{The Compactly Generated Groth. Topology}

Recall from section \ref{subsec:cgh} that if $Y$ is a compactly generated Hausdorff space and $\X$ a topological stack, then one might be tempted to claim that

$$\Map\left(Y,\X\right):= \underset{K_\alpha \hookrightarrow Y} \hl \Map\left(K_\alpha,\X\right),$$
where the weak limit is taken over all compact subsets $K_\alpha$ of $Y$. However, there are some immediate problems with this temptation:

\begin{itemize}
	\item This weak limit may not exist as a topological stack.
	\item The fact that $Y$ is the colimit of its compact subsets in $\cgh$ does not imply that $Y$ is the weak colimit of its compact subsets as a topological stack since the Yoneda embedding does not preserve arbitrary colimits.
\end{itemize}

Recall however that for an arbitrary subcanonical Grothendieck site $\left(\C,J\right)$, the Yoneda embedding $y:\C \hookrightarrow Sh_J\left(\C\right)$ preserves colimits of the form $$C = \underset{C_\alpha \to C} \varinjlim C_\alpha$$ where $\left(C_\alpha \stackrel{f_\alpha}{\rightarrow} C\right)$ is a $J$-cover. We therefore shall construct a Grothendieck topology $\cgt$ on $\cgh$, called the compactly generated Grothendieck topology,  such that for all $Y$, the inclusion of all compact subsets $\left(K_\alpha \hookrightarrow Y\right)$ is a $\cgt$-cover. As it shall turn out, in addition to being Cartesian closed, the bicategory of presentable stacks for this Grothendieck topology will also be complete.

In this subsection, we give a geometric construction of the compactly generated Grothendieck topology on $\cgh$. Those readers not familiar with topos theory may wish to skip to Definition \ref{dfn:cgcover} for the concrete definition of a $\cgt$-cover. Some important properties of $\cgt$-covers are summarized as follows:

\begin{itemize}
	\item[i)] Every open cover is a $\cgt$-cover (Proposition \ref{prop:opcov})
	\item[ii)] For any space, the inclusion of all its compact subsets is a $\cgt$-cover (Corollary \ref{cor:comcov})
	\item[iii)] Every $\cgt$-cover of a locally compact Hausdorff space can be refined by an open one (Proposition \ref{prop:loccov})
	\item[iv)] The category of $\cgt$-sheaves over compactly generated Hausdorff spaces is equivalent to the category of ordinary sheaves over compact Hausdorff spaces (Theorem \ref{thm:eqlshvs}).
\end{itemize}
\subsection{The Compactly Generated Grothendieck Topology}\label{sec:Groth}
Let $$j:\kp \hookrightarrow \cg$$ be the full and faithful inclusion of compact Hausdorff spaces into compactly generated Hausdorff spaces. This induces a geometric morphism $\left(\ju,\jd\right)$

$$\xymatrix{\Skp \ar@<-.5ex>_-\jd[r] & \Sg \ar@<-.5ex>[l]_-\ju}$$
which is an embedding (i.e. $\jd$ is full and faithful \cite{Ieke})\footnote{Technically, these categories are not well defined due to set-theoretic issues, however, this can be overcame by careful use of Grothendieck universes. We will not dwell on this and all such similar size issues in this paper.}.

Denote by

$$\yk:\cg \to \Skp$$
the functor which assigns a compactly generated Hausdorff space $X$ the presheaf $T \mapsto \Hom_{\cg}\left(T,X\right)$ and by

$$\ycg:\kp \to \Sg$$
the functor which assigns a $T \in \kp$ the presheaf $X \mapsto \Hom_{\cg}\left(X,T\right)$. Note that $\yk$ is a fully faithful embedding. The pair $\left(\ju,\jd\right)$ can be constructed as the adjoint pair induced by left Kan extending $\yk$ along the Yoneda embedding. Explicitly:

\begin{equation*}
\ju\left(F\right)\left(T\right)=F\left(T\right)
\end{equation*}
and
\begin{equation*}
\jd\left(G\right)\left(X\right)=\Hom_{\Skp}\left(\yk\left(X\right),G\right).
\end{equation*}

From the general theory of adjoint functors, $\jd$ restricts to an equivalence between, on one hand the full subcategory of $\Skp$ whose objects are those for which the co-unit $\varepsilon\left(j\right)$ is an isomorphism, and on the other hand, the full subcategory of $\Sg$ whose objects are those for which the unit $\eta\left(j\right)$ is an isomorphism. However, since $j$ is fully faithful, the co-unit is always an isomorphism, which verifies that $\jd$ is fully faithful and gives us a way of describing its essential image.

In fact, $\ju$ also has a left adjoint $\js$. The adjoint pair $\js \rightt \ju$ is the one induced by $\ycg$. Hence, $\js$ is the left Kan extension of $\ycg$. We conclude that $\js$ is also fully faithful (see: \cite{Ieke}).

Denote by $Sh\left(\cg\right)$ the topos of sheaves on compactly generated Hausdorff spaces with respect to the open cover topology. We define $Sh\left(\kp\right)$ as the unique topos fitting in the following pullback diagram:

$$\xymatrix{\raise1.0ex\hbox{$Sh\left(\kp\right)$}  \hspace{0.3cm} \ar@{>->}[r] \ar@{>->}[d] & \raise1.0ex\hbox{$\Skp$}  \ar@{>->}[d]\\
Sh\left(\cg\right) \hspace{0.3cm} \ar@{>->}[r] & \Sg}\\$$

Due to the factorization theorem of geometric morphisms in topos theory \cite{Ieke}, the geometric embedding $Sh\left(\kp\right) \rightarrowtail \Skp$ corresponds to a Grothendieck topology $\kpt$ on $\ch$. It is easy to verify that since the functor $j$ is fully faithful, the covering sieves in $\kpt$ for a compact Hausdorff space $K$ are precisely those subobjects $$S \rightarrowtail y\left(K\right)$$ which are obtained by restricting a covering sieve of $K$ with respect to the open cover topology on $\cg$ to $\kp$ via the functor $\ju$. In this sense, the covering sieves are ``the same as in the open cover topology''.

\begin{prop}
The Grothendieck topology $\kpt$ on $\kp$ has a basis of finite covers of the form $\left(T_i \hookrightarrow T\right)_{i=1}^{n}$ by compact neighborhoods (i.e. their interiors form an open cover).
\end{prop}

\begin{proof}
Compact Hausdorff spaces are locally compact in the strong sense that every point has a local base of compact neighborhoods. Hence covers by compact neighborhoods generate the same sieves as open covers.
\end{proof}

Consider the geometric embedding

$$\xymatrix{Sh\left(\kp\right) \ar@<-0.5ex>[r]_{i} & \Skp \ar@<-0.5ex>[l]_{a}},$$
where $a$ denotes the sheafification with respect to $\kpt$.

\begin{rmk}
It is clear that for any presheaf $F$ in $\Sg$, the $\kpt$-sheafification of the restriction of $F$ to $\ch$ is the same as the restriction of the sheafification of F.
\end{rmk}

By composition, we get an embedding of topoi

$$\xymatrix@C=1.5cm{Sh\left(\kp\right) \ar@<-0.5ex>[r]_{\jd \circ i} & \Sg \ar@<-0.5ex>[l]_{a \circ \ju}}\label{eq:emb}.$$
Again by the factorization theorem \cite{Ieke}, there exists a unique Grothendieck topology $\mathscr{CG}$ on $\cg$ such that the category of sheaves $Sh_{\mathscr{CG}}\left(\cg\right)$ is $\jd \left(Sh\left(\kp\right)\right)$. We will construct it and give some of its properties.

There is a very general construction \cite{Ieke} that shows how to extract the unique Grothendieck topology corresponding to this embedding.

First, we define a universal closure operation on $\Sg$. (For details, see \cite{Ieke} V.1 on Lawvere-Tierney topologies.) Let $F$ be a presheaf over $\cg$ and let $m:A \rightarrowtail F$ be a representative for a subobject $A$ of $F$. Then a representative for the subobject $\bar A$ is given by the left hand side of the following pullback diagram

$$\xymatrix{\bar A \ar[d] \ar[r] & \jd a \ju\left(A\right) \ar^-{\jd a \ju\left(m\right)}[d]\\
F \ar^-{\eta_F}[r] & {\jd a \ju\left(F\right),}}$$
where $\eta$ is the unit of the adjunction $a \circ \ju \rightt \jd \circ i$.

To describe the covering sieves of $\cgt$, it suffices to describe the universal closure operation on representables.

Let $X$ be a compactly generated Hausdorff space.

\begin{clm}
The unit $\eta_X$ is an isomorphism.
\end{clm}

\begin{proof}
The restriction $\ju X$ is a $\kpt$-sheaf. Hence $$\jd a \ju X \cong \jd \yk\left(X\right).$$ Furthermore, for any compactly generated Hausdorff space $Y$, $$\jd \yk\left(X\right)\left(Y\right)=\Hom_{\Skp}\left(\yk\left(Y\right),\yk\left(X\right)\right)\cong\Hom_{\cg}\left(Y,X\right)$$ since $\yk$ is fully faithful.
\end{proof}

Now, let $m:A \rightarrowtail X$ a sieve. Then, since the unit $\eta_X$ is an isomorphism, $\bar A$ is represented by the monomorphism $$\jd a \ju\left(A\right) \rightarrowtail X.$$

The covering sieves in $\cgt$ of $X$ are exactly those sieves on $X$ whose closure is equal to the maximal sieve, i.e. $X$. So $$m:A \rightarrowtail X$$ is a covering sieve if and only if $$\bar m:\jd a \ju\left(A\right) \to \jd\ju \left(X\right)\cong X$$ is an isomorphism. Since $\jd$ is fully faithful, this is if and only if $$\tilde m: a \ju\left(A\right) \to \ju X$$ is an isomorphism. In other words, $$m:A \rightarrowtail X$$ is a covering sieve if and only if the $\kpt$-sheafification of $\ju\left(A\right)$ is isomorphic to $\yk\left(X\right)$

\begin{dfn}\label{dfn:cgcover}
Let $X$ be a compactly generated Hausdorff space and let $$\left(\alpha_i:V_i \hookrightarrow X\right)_{i \in I}$$ be family of inclusions of subsets $V_i$ of $X$. Such a family is called a $\cgt$-\textbf{cover} if for any compact subset $K$ of $X$, there exists a (finite) subset $J\left(K\right) \subseteq I$ such that the collection $\left(V_j \cap  K\right)_{j \in J\left(K\right)}$ can by refined by an open cover of $K$. Denote the set of $\cgt$-covers of $X$ by $\mathcal{B}\left(X\right).$
\end{dfn}

\begin{lem}
$\mathcal{B}$ is a basis for the Grothendieck topology $\cgt$.
\end{lem}

\begin{proof}

Let $\left(f_i:Q_i \to X\right)$ be a class of maps into $X$. We denote the sieve it generates by $S_f$. For any compactly generated Hausdorff space $Y$, we have

$$S_f\left(Y\right)=\left\{ {h:Y\to X\mbox{ such that }h\mbox{ factors through }f_i\mbox{ for some }i } \right\}.$$
So, $S_f$ is a covering sieve if and only if when restricted to $\kp$, its $\kpt$-sheafification is isomorphic to $\yk\left(X\right)$. We first note that $\ju\left(S_f\right)$ is clearly a $\kpt$-separated presheaf. Hence, its sheafification is the same as $\ju\left(S_f\right)^{+}$. Since $\yk\left(X\right)$ is a sheaf, the canonical map $$\ju\left(S_f\right) \rightarrowtail \yk\left(X\right)$$ factors uniquely as

$$\xymatrix{\ju\left(S_f\right) \hspace{1mm} \ar[d] \ar@{>->}[r] & \yk\left(X\right)\\
\ju\left(S_f\right)^{+} \ar@{>->}[ur] &}.$$
It suffices to see when the map $\ju\left(S_f\right)^{+} \rightarrowtail \yk\left(X\right)$ is an epimorphism.

Let $\tilde S_f$ be the presheaf on $\kp$

$$\tilde S_f \left(T\right)=\left\{\mathcal{U}=\left(U_i\right)_{j=1}^{n}, \left(a_j \in S_f\left(U_j\right)\right)_{j=1}^{n} \:\: | \:\: a_i|_{U_{ij}}=a_j|_{U_{ij}} \mbox{ for all i,j} \right\}.$$
Then the map $\ju\left(S_f\right)^{+} \rightarrowtail \yk\left(X\right)$ fits in a diagram
$$\xymatrix{\tilde S_f \ar@{->>}[d] \ar[r] & \yk\left(X\right)\\
\ju\left(S_f\right)^{+} \ar@{>->}[ur] &}.$$
It suffices to see when the canonical map $\tilde S_f \to \yk\left(X\right)$ is point-wise surjective. This is precisely when for any map $h:K \to X$ from a compact Hausdorff space $K \in \kp$, there exists an open cover $\left(U_j\right)_j$ of $K$ such that for all $j$, $h|_{U_j}$ factors through $f_i$ for some $i$. Classes of maps with codomain $X$ with this property constitute a large basis for the Grothendieck topology $\cgt$. It is in fact maximal in the sense that $S$ is a covering sieve if and only if it is one generated by a large cover of this form. We will now show that any such large covering family has a refinement by one of the form of the lemma.

Let $\left(f_i:Q_i \to X\right)$ denote such a (possibly large) family and let $$\left(i_\alpha:K_\alpha \hookrightarrow X\right)$$ denote the inclusion of all compact subsets of $X$. Then for each $\alpha$, there exists a finite open cover of $K_\alpha$, $\left(\mathcal{O}_j^\alpha\right),$ such that the inclusion of each $\mathcal{O}_j^\alpha$ into $X$ factors through some $$f_{j,\alpha}:Q_{j,\alpha} \to X$$ via a map $$\lambda_j^\alpha:\mathcal{O}_j^\alpha \to Q_{j,\alpha}.$$
Let $\mathcal{U}:=\left(\mathcal{O}_j^\alpha \hookrightarrow X\right)_{j,\alpha}$. Let $g:L \to X$ be a map with $L \in \kp$. Then $g\left(L\right)=K_\alpha$ for some $\alpha$. Let $$\mathcal{V}_L^{g}:=\left(g^{-1}\left(\mathcal{O}_j^\alpha\right)\right).$$ Then $\mathcal{V}_L^{g}$ is an open cover of $L$ such that the restriction of $g$ to any element of the cover factors through the inclusion of some $\mathcal{O}_j^\alpha$ into $X$. Hence the sieve generated by $\mathcal{U}$ is a covering sieve for $\cgt$ which refines the sieve generated by $\left(f_i:Q_i \to X\right)$.

\end{proof}

We have the following obvious proposition whose converse is not true:

\begin{prop}\label{prop:opcov}
Any open cover of a space is also a $\cgt$ cover.
\end{prop}

In particular, one cover that is quite useful is the following.

\begin{cor}\label{cor:comcov}
For any compactly generated Hausdorff space, the inclusion of all compact subsets is a $\cgt$-cover.
\end{cor}

However, for the category $\lk$ of locally compact Hausdorff spaces, it suffices to work with open covers:

\begin{prop}
Every $\cgt$-cover of a locally compact Hausdorff space $X \in \lk$ can be refined by an open covering.
\label{p:lcc}
\end{prop}

\begin{proof}
Let $X \in \lk$ and let $\mathcal{V}=\left(\alpha_i:V_i \hookrightarrow X\right)_{i\in I}$ be a $\cgt$-cover of $X$. Let $\left(K_l\right)$ be a topological covering of $X$ by compact subsets such that the interiors $int\left(K_l\right)$ constitute an open cover for $X$. Then for each $K_l$, there exists a finite subset $J\left(K_l\right) \subseteq I$ such that

$$\left(V_j \cap K_l\right)_{j \in J\left(K_l\right)}$$
can be refined by an open cover $\left(W_j\right)_{j \in J\left(K_l\right)}$ for $K_l$ such that the inclusion of each $W_j$ into $X$ factors through the inclusion of $V_j$. Let $\mathcal{U}:=\left(W_j\right)_{l, j \in J\left(K_l\right)}$. Then $\mathcal{U}$ is an open cover of $X$ refining $\mathcal{V}$.
\end{proof}

We can now define the $\cgt$-sheafification functor $a_{\cgt}$ either by the covering sieves, or by using the basis $\mathcal{B}$ (i.e. both will give naturally isomorphic functors). Let $\Sgs$ denote the category of $\cgt$-sheaves. Then we have an embedding of topoi given by

$$\xymatrix@C=1.9cm{\Sgs \ar@<-0.5ex>[r]_{\ell} & \Sg \ar@<-0.5ex>[l]_{a_{\cgt}}},$$
where $\ell:\Sgs \hookrightarrow \Sg$ is the inclusion of the category of sheaves.

By the previous observation that open covers are $\cgt$-covers, we also have

$$\Sgs \subset Sh\left(\cg\right),$$
where $Sh\left(\cg\right)$ is the category of sheaves on $\cg$ with respect to the open cover topology.

By construction, we have the following theorem:

\begin{thm}\label{thm:eqlshvs}
There is an equivalence of topoi

$$\xymatrix@C=1.9cm{Sh\left(\kp\right) \ar@<-0.5ex>[r]_(0.45){a_{\cgt}\circ\jd \circ i} & \Sgs \ar@<-0.5ex>[l]_(0.558){a \circ \ju \circ \ell}}$$
such that
$$\xymatrix@C=1.9cm{Sh\left(\kp\right) \ar@<-0.5ex>[r]_(0.45){a_{\cgt}\circ\jd \circ i} & \Sgs \ar@<-0.5ex>[l]_(0.558){a \circ \ju \circ \ell} \ar@<-0.5ex>[r]_(0.55){\ell} & \Sg \ar@<-0.5ex>[l]_(0.45){a_{\cgt}}}$$
is a factorization of
$$\xymatrix@C=1.5cm{Sh\left(\kp\right) \ar@<-0.5ex>[r]_{\jd \circ i} & \Sg \ar@<-0.5ex>[l]_{a \circ \ju}}$$
(up to natural isomorphism).
\end{thm}

Note that the essential image of $\ell$ is the same as the essential image of $\jd \circ i$. Hence, a presheaf $F$ in $\Sg$ is a $\cgt$-sheaf if and only if the unit $\eta$ of $a \circ \ju \rightt \jd \circ i$ is an isomorphism at $F$. We have the following immediate corollary.

\begin{cor}
The Grothendieck topology $\cgt$ is subcanonical.
\label{cor:subcan}
\end{cor}

\begin{lem}
If $F$ is a $\cgt$-sheaf, then $\ju \left(F\right)$ is a $\kpt$-sheaf.
\end{lem}

\begin{proof}
Since $F$ is a $\cgt$-sheaf, it is in the essential image of $\jd \circ i$, hence, via $\eta_F$, we have $$F \cong \jd a \ju \left(F\right).$$
By applying $\ju$ we have
$$\ju\left(F\right) \cong \ju \jd a \ju \left(F\right).$$
Since the co-unit $\varepsilon\left(j\right)$ of $\ju \rightt \jd$ is an isomorphism, this yields
$$\ju\left(F\right)\cong a \ju \left(F\right).$$
\end{proof}

\begin{cor}
If $F$ is a presheaf in $\Sg$,

\begin{equation*}
a_{\cgt}\left(F\right)\cong \jd a \left(\ju F\right).
\end{equation*}

If $F \in Sh\left(\cg\right)$ is a sheaf in the open cover topology then its $\cgt$-sheafification is given by $\jd \ju F$.
\end{cor}

\begin{proof}
$$\jd a \left(\ju F\right)\cong \left(\jd \circ i\right) \circ a \left(\ju F\right),$$
so $\jd a \left(\ju F\right)$ is in the image of $\jd \circ i$, hence a $\cgt$-sheaf. Now, let $G$ be any $\cgt$-sheaf. Then:

\begin{eqnarray*}
\Hom_{\Sg}\left(\jd a \left(\ju F\right), G\right) &\cong& \Hom_{\Skp}\left(a \ju \left(F\right), \ju\left(G\right)\right)\\
&\cong& \Hom_{\Skp}\left(\ju \left(F\right), \ju\left(G\right)\right)\\
&\cong& \Hom_{\Sg}\left(F,\jd\ju\left(G\right)\right)\\
&\cong& \Hom_{\Sg}\left(F,G\right).
\end{eqnarray*}

\end{proof}

We end this subsection by noting ordinary sheaves and $\cgt$-sheaves agree on locally compact Hausdorff spaces:

Let $\varsigma$ denote the unit of the adjunction $\ju \rightt \jd$.

\begin{prop}\label{prop:loccov}
Let $F \in Sh\left(\cg\right)$ be a sheaf in the open cover topology and $X$ in $\lk$ a locally compact Hausdorff space. Then the map
$$\varsigma_F\left(X\right):F\left(X\right) \to \jd \ju F\left(X\right) \cong a_{\cgt}\left(F\right)\left(X\right)$$
is a bijection. In particular, $F$ and $a_{\cgt}\left(F\right)$ agree on locally compact Hausdorff spaces.
\end{prop}

\begin{proof}
This follows immediately from Proposition \ref{p:lcc}
\end{proof}

\subsection{Stacks for the Compactly Generated Grothendieck Topology}
Denote by $\hat\yk$ the $2$-functor
$$\hat\yk:\cg \to \Psk$$
induced by the inclusion $$\left(\mspace{2mu} \cdot \mspace{2mu}\right)^{(id)}:\Skp \hookrightarrow \Psk.$$
Then, it produces a $2$-adjoint pair, which we will denote by $\hat\ju \rightt \hat\jd$, by constructing $\hat\ju$ as the weak left Kan extension of $\hat\yk$, and by letting
$$\hat\jd\Y\left(X\right):=\Hom_{\Psk}\left(\hat\yk\left(X\right),\Y\right).$$
By setting
$$\hat\ju\left(\X\right)(T):=\X(T),$$
we get a $2$-functor which is weak colimit preserving and whose restriction to representables is the same as $\hat\yk$, hence, by uniqueness, the above equation for $\hat\ju$ must be correct. Note that the co-unit is an equivalence, hence $\hat\jd$ is fully faithful.
Similarly, denote again by $\hat\ycg$ the $2$-functor
$$\hat\ycg:\kp \to \Psg$$
induced by the inclusion $$\left(\mspace{2mu} \cdot \mspace{2mu}\right)^{(id)}:\Sg \hookrightarrow \Psg.$$
The weak left Kan extension of the $\hat\ycg$ along the Yoneda embedding, just as before, is left $2$-adjoint to $\ju$ and $2$-categorically full and faithful.

Note that for $F \in \Sg$ we have $$\hat \ju \left(F\right) \simeq \left(\ju\left(F\right)\right)^{id},$$ and similarly for $G \in \Skp,$ $$\hat \jd \left(G\right) \simeq \left(\jd\left(G\right)\right)^{id}.$$ Hence, to ease notation, we shall from now on denote $\hat\ju$ simply by $\ju$ and similarly for $\hat\jd.$

Let $\Stk$ denote the bicategory of stacks on $\kp$ with respect to the Grothendieck topology $\kpt$. Then we have a $2$-adjoint pair $a \rightt i$
$$\xymatrix{\Stk \ar@<-0.5ex>[r]_{i} & \Psk \ar@<-0.5ex>[l]_{a}},$$
where $a$ is the stackification $2$-functor (Definition \ref{dfn:stackification}) and $i$ is the inclusion. Then, by composition, we get a $2$-adjoint pair $a \circ \ju \rightt \jd \circ i$
$$\xymatrix@C=1.5cm{\Stk \ar@<-0.5ex>[r]_{\jd \circ i} & \Psg \ar@<-0.5ex>[l]_{a \circ \ju}}.$$

\begin{dfn}
A stack with respect to the Grothendieck topology $\cgt$ on $\cgh$ will be called a $\cgt$-stack.
\end{dfn}

Let $\Stcg$ denote the full sub-bicategory of $\Psg$ consisting of $\cgt$-stacks, and let $a_{\cgt}$ denote the associated stackification $2$-functor, and $\ell$ the inclusion, so $a_{\cgt} \rightt \ell$.

Just as before, since every open covering is a $\cgt$-cover,

$$\Stcg \subset \Stc.$$
The following results and their proofs follow naturally from those of the previous section when combined with the Comparison Lemma for stacks, a straightforward stacky analogue of the theorem in \cite{SGA} III:

\begin{cor}\label{cor:eql2tps}
There is an equivalence of bicategories
$$\xymatrix@C=1.9cm{\Stk \ar@<-0.5ex>[r]_(0.45){a_{\cgt} \circ \jd \circ i} & \Stcg \ar@<-0.5ex>[l]_(0.558){a \circ \ju \circ \ell}},$$
such that
$$\xymatrix@C=1.9cm{\Stk \ar@<-0.5ex>[r]_(0.45){a_{\cgt}\circ\jd \circ i} & \Stcg \ar@<-0.5ex>[l]_(0.558){a \circ \ju \circ \ell} \ar@<-0.5ex>[r]_(0.55){\ell} & \Psg \ar@<-0.5ex>[l]_(0.45){a_{\cgt}}}$$
is a factorization of
$$\xymatrix@C=1.5cm{\Stk \ar@<-0.5ex>[r]_{\jd \circ i} & \Psg \ar@<-0.5ex>[l]_{a \circ \ju}}$$
\label{thm:eqst}
(up to natural equivalence).
\end{cor}

\begin{lem}
If $\X$ is a $\cgt$-stack, then $\ju \left(\X\right)$ is a $\kpt$-stack.
\end{lem}

\begin{cor}
If $\X$ is a weak presheaf in $\Psg$,

\begin{equation*}
a_{\cgt}\left(\X\right)\simeq \jd a \left(\ju \X\right).
\end{equation*}
If $\X \in St\left(\cg\right)$ is a stack in the open cover topology then its $\cgt$-stackification is given by $\jd \ju \X$.
\end{cor}

Again, let $\varsigma$ denote the unit of the $2$-adjunction $\ju \rightt \jd$.

\begin{prop}
Let $\X \in St\left(\cg\right)$ be a stack in the open cover topology and $X$ in $\lk$ a locally compact Hausdorff space. Then the map
$$\varsigma_{\X}\left(X\right):\X\left(X\right) \to \jd \ju \X\left(X\right) \simeq a_{\cgt}\left(\X\right)\left(X\right)$$
is an equivalence of groupoids. In particular, $\X$ and $a_{\cgt}\left(\X\right)$ agree on locally compact Hausdorff spaces.
\label{prop:eqparat}
\end{prop}

\section{Compactly Generated Stacks}\label{sec:cgst}
\subsection{Compactly Generated Stacks}\label{sec:tcgs}
\begin{dfn}\index{compactly generated stack}
A \textbf{compactly generated stack} is a presentable $\cgt$-stack (see Definition \ref{dfn:presentable}).
\end{dfn}

We denote the full sub-bicategory of $\Stcg$ of compactly generated stacks by $\Ctst$.

Intrinsically, a compactly generated stack is a $\cgt$-stack $\X$ such that there exists a compactly generated Hausdorff space $X$ and a representable $\cgt$-epimorphism
$$p:X \to \X.$$
The map above is a $\cgt$-\textbf{atlas} for $\X$.

Let $$\tilde y_{\kp}:\cg Gpd \to \Psk$$ denote the $2$-functor
$$\G \mapsto \Hom_{\cg Gpd} \left(\left(\mspace{2mu} \cdot \mspace{2mu}\right)^{id},\G\right).$$
Given a topological groupoid $\G$ in $\cg Gpd$, denote by $\left[\G\right]_{\kpt}$ the associated $\kpt$-stack $a \circ \tilde y_{\kp}\left(\G\right)$.

Let $\Ctst'$ denote the essential image in $\Psk$ of $a \circ \tilde y_{\kp}$, i.e., it is the full sub-bicategory of consisting of $\kpt$-stacks equivalent to $\left[\G\right]_{\kpt}$ for some compactly generated topological groupoid $\G$. It is immediate from Theorem \ref{thm:eqst} that this bicategory is equivalent to $\Ctst$. In fact, the functor $\jd$ restricts to an equivalence $$\jd|_{\Ctst'}:\Ctst' \to \Ctst$$ of bicategories. Hence we have proven:

\begin{thm}\label{thm:req}
The bicategory of compactly generated stacks, $\Ctst$, is equivalent to the essential image of
$$\ju|_{\Tst}: \Tst \to St\left(\ch\right).$$
\end{thm}

Note that from Theorems \ref{thm:bunmor} and \ref{thm:bicatp1}, $\Ctst$ is also equivalent to the bicategory of fractions $\cg Gpd \left[W_{\cgt}^{-1}\right]$ of compactly generated Hausdorff topological groupoids with inverted $\cgt$-Morita equivalences, and also to the bicategory $Bun^{\cgt} \cg Gpd$ of compactly generated Hausdorff topological groupoids with left $\cgt$-principal bundles as morphisms:

\begin{thm}
The $2$-functor
$$a_{\cgt} \circ \tilde y:\cg Gpd \to \Ctst$$
induces an equivalence of bicategories
$$P_{\cgt}:\cg Gpd \left[W_{\cgt}^{-1}\right] \stackrel{\sim}{\longrightarrow} \Ctst.$$
\label{thm:bicatcg1}
\end{thm}

\begin{thm}
The $2$-functor
$$a_{\cgt} \circ \tilde y:\cg Gpd \to \Ctst$$
induces an equivalence of bicategories
$$P_{\cgt}':Bun^{\cgt} \cg Gpd \stackrel{\sim}{\longrightarrow} \Ctst.$$
\label{thm:bicatcg2}
\end{thm}

We note that the principal bundles in $Bun^{\cgt} \cg Gpd$ have a very simple description:

Recall that our notion of principal bundle depends on a Grothendieck topology. When the projection map of a principal bundle admits local sections (with respect to the open cover topology), it is called \textbf{ordinary}.

\begin{prop}
If $\G$ is a topological groupoid in $\cg Gpd$, $X$ is a compactly generated Hausdorff space, and
$$\xymatrix @R=2pc @C=0.15pc {\G_1  \ar@<+.7ex>[d] \ar@<-.7ex>[d] & \acts & P \ar_{\mu}[lld] \ar^{\pi}[d] \\
\G_0 && X}$$
is a left $\G$-space over $\pi$, then it is a $\cgt$-principal bundle if and only if the restriction of $P$ to $K$ is an ordinary principal $\G$-bundle over $K$, for every compact subset $K \subseteq X$.
\label{prop:cgpb}
\end{prop}

\begin{proof}
Suppose that we are given a left $\G$-space $\pi:P \to X,$ whose restriction to $K$ is an ordinary principal $\G$-bundle over $K$, for every compact subset $K \subseteq X$. Then, for each compact subset $K_\alpha \subseteq X$, we can choose an open cover $\left(U_\alpha^{i} \hookrightarrow K_\alpha\right)_{i=1}^{N_\alpha}$ over which $P$ admits local sections. Then $P$ admits local sections with respect to the $\cgt$-cover
$$\mathcal{U}:=\left(U_\alpha^{i} \hookrightarrow X\right).$$
The converse is trivial.
\end{proof}

\begin{cor}\label{cor:tiredxx}
If $\G$ and $\h$ are topological groupoids in $\cg Gpd$ and $\h_0$ is locally compact Hausdorff, then

\begin{equation*}
Bun_{\G}^{\cgt}\left(\h\right) \simeq Bun_{\G}\left(\h\right),
\end{equation*}
where $Bun_{\G}^{\cgt}\left(\h\right)$ denotes the groupoid of $\cgt$-principal $\G$-bundles over $\h$, and $Bun_{\G}\left(\h\right)$ denotes the groupoid of ordinary principal $\G$-bundles over $\h$.

Equivalently:

If $\X$ and $\Y$ are topological stacks, and $\Y$ admits a locally compact Hausdorff atlas $Y \to \Y$, then the map
$$\Hom_{\St\left(\cg\right)}\left(\Y,\X\right) \to \Hom_{\St\left(\cg\right)}\left(a_{\cgt}\left(\Y\right),a_{\cgt}\left(\X\right)\right)$$
induced by the unit $\varsigma_{\X}:\X \to a_{\cgt}\left(\X\right)$, and the $2$-adjunction $a_{\cgt} \rightt \ell$ is an equivalence of groupoids.
\label{cor:locsame}
\end{cor}

\begin{cor}\label{cor:fflk}
The $\cgt$-stackification functor restricted to the sub-bicategory of topological stacks consisting of those topological stacks which admit a locally compact Hausdorff atlas, is $2$-categorically full and faithful.
\end{cor}

\begin{thm}\label{thm:subcat}
The bicategory of compactly generated stacks is equivalent to the sub-bicategory of topological stacks consisting of those topological stacks which admit a locally compact Hausdorff atlas.\footnote{When we work over compactly generated Hausdorff spaces}
\end{thm}

\begin{proof}
Denote the sub-bicategory of topological stacks consisting of those topological stacks which admit a locally compact Hausdorff atlas by $\Tst_{\lk}.$ Note that the image of $$a_{\cgt}|_{\Tst_{\lk}}:\Tst_{\lk} \to \St_{\cgt}\left(\cgh\right)$$ lies entirely in $\Ctst.$ By Corollary \ref{cor:fflk} this $2$-functor is full and faithful. It suffices to show it is essentially surjective. Notice that the essential image is those compactly generated stacks which admit a locally compact Hausdorff atlas. To complete the proof, note that if $X \to \X$ is any atlas of a compactly generated stack, then the inclusion of all compact subsets of $X$ is a $\cgt$-cover, hence $$\underset{\alpha} \coprod K_\alpha \to X \to \X$$ is a $\cgt$-atlas for $\X$ which is locally compact Hausdorff.
\end{proof}

\begin{thm}
The bicategory of compactly generated stacks has arbitrary products.
\label{thm:allprod}
\end{thm}

\begin{proof}
Let $\X_i$ be an arbitrary family of compactly generated stacks. Then we can choose topological groupoids $\G_i$ in $\cg Gpd$ such that

$$\X_i \simeq \left[\G_i\right]_{\cgt}.$$
Note that

\begin{equation*}
\left[\G_i\right]_{\cgt} \simeq \jd \left[\G_i\right]_{\kpt}.
\end{equation*}
In light of Lemma \ref{lem:fibm}, we may assume without loss of generality that each $\G_i$ is fibrant. Under this assumption, by Theorem \ref{thm:imp}, it follows that

\begin{equation*}
\left[\G_i\right]_{\cgt} \simeq \jd \tilde y_{\kp}\left(\G_i\right).
\end{equation*}
Note that the product $\prod\limits_i \X_i$ is a $\cgt$-stack, as any bicategory of stacks is complete. It suffices to show that this product is still presentable.

Recall that $\tilde y_{\kp}$ preserves small weak limits. Moreover, $\jd$ does as well as it is a right $2$-adjoint. Hence

\begin{equation*}
\prod\limits_i \X_i \simeq \prod\limits_i {\jd \tilde y_{\kp}\left(\G_i\right)} \simeq \jd \tilde y_{\kp} \left(\prod\limits_i \G_i\right).
\end{equation*}
It follows that
\begin{eqnarray*}
\prod\limits_i \X_i \simeq a_{\cgt} \left( \prod\limits_i \X_i \right) &\simeq& a_{\cgt} \left(\jd \tilde y_{\kp} \left(\prod\limits_i \G_i\right)\right)\\
&\simeq& \left(\jd \circ a \circ \ju\right) \circ \jd \circ \left(\ju \tilde y\right)\left(\prod\limits_i \G_i\right)\\
&\simeq& \left(\jd \circ a \circ \ju\right) \circ \left(\tilde y\right)\left(\prod\limits_i \G_i\right)\\
&\simeq& \left[\prod\limits_i \G_i\right]_{\cgt}.
\end{eqnarray*}
\end{proof}

\begin{cor}
The bicategory of compactly generated stacks is closed under arbitrary small weak limits.
\label{cor:alllimits}
\end{cor}

\begin{proof}
Since $\cg Gpd$ is closed under binary weak fibered products and the stackification $2$-functor $a_{\cgt}$ preserves finite weak limits, the bicategory $\Ctst$ is closed under binary weak fibered products. By Theorem \ref{thm:allprod}, this bicategory has arbitrary small products. Since $\Ctst$ is a $\left(2,1\right)$-category, by \cite{htt} it follows that $\Ctst$ has all limits and hence is complete.
\end{proof}

\subsection{Mapping Stacks of Compactly Generated Stacks}\label{sec:map}

Recall that if $\X$ and $\Y$ are any stacks over $\cg,$ they have a mapping stack
$$\Map\left(\Y,\X\right)\left(T\right)=\Hom_{\Psg}\left(\Y \times T,\X\right).$$
It is the goal of this section to prove that if $\X$ and $\Y$ are compactly generated stacks, then so is $\Map\left(\Y,\X\right)$.

\begin{lem}
If $\Y \simeq \left[\h\right]_{\cgt}$ is a compactly generated stack with $\h_0$ compact Hausdorff, and $\X \simeq \left[\G\right]_{\cgt}$ an arbitrary compactly generated stack, then $\Map\left(\Y,\X\right)$ is a compactly generated stack. More specifically, if $\K$ is a presentation for the topological stack $\Map\left(\left[\h\right],\left[\G\right]\right)$ ensured by Theorem \ref{thm:topcompmapme}, then

$$\Map\left(\Y,\X\right) \simeq \left[\K\right]_\cgt.$$
In particular, $\Map\left(\Y,\X\right)$ and $\Map\left(\left[\h\right],\left[\G\right]\right)$ restrict to the same stack over locally compact Hausdorff spaces.
\label{lem:cpmap}
\end{lem}

\begin{proof}
Since any $\cgt$-stack is completely determined by its restriction to $\kp$, it suffices to show that for any compact Hausdorff space $T \in \kp$, $$\left[\K\right]_{\cgt}\left(T\right) \simeq \Hom_{\Psg}\left(\Y \times T,\X\right).$$
But, since $T$ is compact Hausdorff

\begin{equation*}
\left[\K\right]_{\cgt}\left(T\right) \simeq \left[\K\right]\left(T\right)
\end{equation*}
and because of the definition of $\K$
\begin{equation*}
\left[\K\right]\left(T\right) \simeq \Hom_{\Psg}\left(\left[\h\right] \times T,\left[\G\right]\right).
\end{equation*}
Furthermore, since $\Y \times T \simeq \left[\h \times T\right]_{\cgt}$ and $\h \times T$ has compact Hausdorff object space, by Corollary \ref{cor:locsame},

\begin{eqnarray*}
\Hom_{\Psg}\left(\left[\h\right] \times T,\left[\G\right]\right) &\simeq& \Hom_{\Psg}\left(\left[\h\right]_{\cgt} \times T,\left[\G\right]_{\cgt}\right)\\ &\simeq& \Hom_{\Psg}\left(\Y \times T,\X\right).\\
\end{eqnarray*}
Hence $$\left[\K\right]_{\cgt}\left(T\right) \simeq \Hom\left(\Y \times T,\X\right).$$
\end{proof}

\begin{cor}
If $K$ is a compact Hausdorff space and $\X$ an arbitrary compactly generated stack, then $\Map\left(K,\X\right)$ is a compactly generated stack.
\end{cor}

\begin{lem}\label{lem:spst}
If $X$ is a compactly generated Hausdorff space and $\X$ an arbitrary compactly generated stack, then $\Map\left(X,\X\right)$ is a compactly generated stack.
\end{lem}

\begin{proof}
Let $\left(K_\alpha \stackrel{i_\alpha}{\hookrightarrow} X\right)$ denote the inclusion of all compact subsets of $X$. This is a $\cgt$-cover for $X$. Let $Y$ be an arbitrary compactly generated Hausdorff space. 

By Proposition \ref{covlim}, we have that in $\Ctst$

\begin{equation*}
X \simeq  \underset{K_\alpha \hookrightarrow X} \hc K_\alpha
\end{equation*}
and
\begin{equation*}
X \times Y \simeq  \underset{K_\alpha \times Y \hookrightarrow X \times Y} \hc \left(K_\alpha \times Y\right),
\end{equation*}
the latter since colimits are universal in any bicategory of stacks.
Hence
\begin{eqnarray*}
\Map\left(X,\X\right)\left(Y\right) &\simeq& \Hom_{\Psg}\left(X \times Y,\X\right)\\
&\simeq& \Hom_{\Psg}\left(\underset{K_\alpha \times Y \hookrightarrow X \times Y} \hc \left(K_\alpha \times Y\right),\X\right)\\
&\simeq& \underset{K_\alpha \times Y \hookrightarrow X \times Y} \hl \Hom_{\Psg}\left(K_\alpha \times Y,\X\right)\\
&\simeq& \underset{K_\alpha \times Y \hookrightarrow X \times Y} \hl \Hom_{\Psg}\left(Y,\Map\left(K_\alpha,\X\right)\right)\\
&\simeq& \Hom_{\Psg}\left(Y,\underset{K_\alpha \hookrightarrow X} \hl \Map\left(K_\alpha,\X\right)\right)\\
&\simeq& \left(\underset{K_\alpha \hookrightarrow X} \hl \Map\left(K_\alpha,\X\right)\right)\left(Y\right).\\
\end{eqnarray*}
Therefore
\begin{equation*}
\Map\left(X,\X\right) \simeq \underset{K_\alpha \hookrightarrow X} \hl \Map\left(K_\alpha,\X\right).
\end{equation*}
So by Corollary \ref{cor:alllimits}, $\Map\left(X,\X\right)$ is a compactly generated stack.
\end{proof}

\begin{thm}\index{mapping stack}
If $\X$ and $\Y$ are arbitrary compactly generated stacks, then $\Map\left(\Y,\X\right)$ is a compactly generated stack.
\label{thm:carclos}
\end{thm}

\begin{proof}

Let $\Y$ be presented by a topological groupoid $\h$. By Lemma \ref{lem:319}, we can write $\Y$ as the weak colimit of the following diagram:
$$\xymatrix{ {\h_1 \times _{\h_0}\h_1} \ar@<0.9ex>[r] \ar@<0.0ex>[r] \ar@<-0.9ex>[r] & \h_1 \ar@<+.7ex>[r]\ar@<-.7ex>[r]& \h_0},$$
where the three parallel arrows are the first and second projections and the composition map.

Furthermore, let $X$ be any compactly generated Hausdorff space. Then $\Y \times X$ is the weak colimit of

$$\xymatrix{ {\left(\h_1 \times _{\h_0}\h_1\right) \times X}  \ar@<0.9ex>[r] \ar@<0.0ex>[r] \ar@<-0.9ex>[r] & \h_1 \times X \ar@<+.7ex>[r]\ar@<-.7ex>[r]& \h_0 \times X}.$$
With this in mind, in much the same way as Lemma \ref{lem:spst}, some simple calculations allow one to express $\Map\left(\Y,\X\right)$ as a weak limit of a diagram involving $\Map\left(\h_1 \times _{\h_0} \h_1,\X\right)$, $\Map\left(\h_1,\X\right)$, and $\Map\left(\h_0,\X\right),$ all of which are compactly generated stacks by Lemma \ref{lem:spst}.
\end{proof}

We presented the proof of Cartesian-closure in this way to emphasize the role of completeness and compact generation. We will now give a concrete description of the mapping stack of two compactly generated stacks. Note that since the inclusion of all compact subsets of a space is a $\cgt$-cover, every compactly generated stack has a locally compact, paracompact Hausdorff atlas.
\begin{thm}
Let $\X\simeq \left[\G\right]_{\cgt}$ and $\Y\simeq\left[\h\right]_{\cgt}$ be two compactly generated stacks. Assume (without loss of generality) that $\h_0$ is locally compact and paracompact Hausdorff. Then $$\left[Fib\left(\G\right)^\h\right]_\cgt \simeq \Map\left(\Y,\X\right).$$
\end{thm}
\begin{proof}
It suffices to check that $\left[Fib\left(\G\right)^\h\right]_\cgt$ and $\Map\left(\Y,\X\right)$ agree on every compact Hausdorff space $T$. Following the same proof as Theorem \ref{thm:cw}, one only has to realize that the product of a compact Hausdorff space with a paracompact Hausdorff space is paracompact Hausdorff \cite{paraprod}. The rest of the proof is identical.
\end{proof}

\begin{rmk}
This implies that the bicategory of topological stacks (in compactly generated Hausdorff spaces) with locally compact Hausdorff atlases is Cartesian closed. This might seem surprising since, after all, locally compact Hausdorff spaces are quite far from being Cartesian closed. What is happening is that the exponential of two locally compact Hausdorff spaces is not a space in this description, but a stack! (This stack is actually a sheaf.) In fact, the category of compactly generated Hausdorff spaces embeds into this bicategory by sending a space $X$ to the stack associated to the topological groupoid $\left(X\right)^{\left(id\right)}_{\K}$, where $\K$ denotes the $\cgt$-cover which is the inclusion of all compact subsets of $X$.
\end{rmk}

\subsection{Homotopy Types of Compactly Generated Stacks}\label{sec:homotopytype}

In \cite{NoohiH}, Noohi constructs a functorial assignment to each topological stack $\X,$ a weak homotopy type. For $\X$ a topological stack, its weak homotopy type turns out to be the weak homotopy type of $\left\|\G\right\|$ for any topological groupoid $\G$ for which $\X \simeq \left[\G\right]$. Moreover, each topological stack $\X$ admits at atlas which is also a weak homotopy equivalence; the canonical atlas $$\varphi:\left\|\G\right\| \to \X$$ coming from Lemma \ref{lem:fibm} is a weak homotopy equivalence.

A particular corollary is:

\begin{cor}
If $\G \to \h$ is a Morita equivalence, the induced map $$\left\|\G\right\| \to \left\|\h\right\|$$ is a weak homotopy equivalence.
\end{cor}

This is a classical result. For instance, it is proven for the case of \'etale topological groupoids in \cite{Class} and \cite{Weak}. \footnote{I would like to thank Ieke Moerdijk for explaining to me how to extend his method of proof to any topological groupoid whose object and arrow spaces have a basis of contractible open sets.}

In this section, we extend these results to the setting of compactly generated stacks. We begin with a technical notion of a shrinkable map, which will prove quite useful.

\begin{dfn}\cite{Dold}\index{shrinkable map}
A continuous map $f:X \to B$ is \textbf{shrinkable} if admits a section $s:B \to X$ together with a homotopy $$H:I \times X \to X$$ from $sf$ to $id_X$ over $B,$ i.e. for all $t$, the map $$H_t:X \to X$$ is a map in $Top/B$ from $f$ to $f$.
\end{dfn}

\begin{rmk}
Every shrinkable map is in particular a homotopy equivalence.
\end{rmk}

\begin{dfn}
A continuous map $f:X \to B$ is \textbf{locally shrinkable} \cite{NoohiH} if there exists an open cover $\left(U_\alpha\right)$ of $B$ such that for each $\alpha,$ the induced map $$f|_{U_{\alpha}}:f^{-1}\left(U_{\alpha}\right) \to U_{\alpha}$$ is shrinkable. A map $f:X \to B$ is called $\cgt$-locally shrinkable if there exists a $\cgt$-cover $\left(V_i\right)$ of $B$ such that the same condition holds.
\end{dfn}

Clearly, shrinkable $\Rightarrow$  locally shrinkable $\Rightarrow$ $\cgt$-locally shrinkable.

\begin{dfn}\index{quasi-shrinkable}
A continuous map $f:X \to B$ is \textbf{quasi-shrinkable} if for every map $T \to B$ from a locally compact, paracompact Hausdorff space $T,$ the induced map $$X \times_{Y} T \to T$$ is shrinkable.
\end{dfn}

\begin{lem}\label{lem:cgqsh}
Every $\cgt$-locally shrinkable map is quasi-shrinkable.
\end{lem}
\begin{proof}
Since every $\cgt$-cover of a locally compact Hausdorff space can be refined by an open one, and every open cover of a paracompact Hausdorff space can be refined by a numerable one, the result follows from \cite{Dold}, Corollary 3.2.
\end{proof}

\begin{dfn}\index{universal weak equivalence}
A map $f:X \to Y$ of spaces is a \textbf{universal weak equivalence} if for any map $T \to Y$, the induced map $T \times_{Y} X \to T$ is a weak homotopy equivalence.
\end{dfn}

The following lemma is a useful characterization of universal weak equivalences:

\begin{lem}
A map $f:X \to B$ is a universal weak equivalence if and only if, for all $n \ge 0,$ for any map $D^n \to B$ from the $n$-disk, the induced map $X \times_B D^n \to D^n$ is a weak homotopy equivalence (i.e. $X \times_B D^n$ is weakly contractible).
\end{lem}

\begin{proof}
One direction is clear by definition.

Conversely, let $f:X \to B$ be a map satisfying the stated hypothesis for each $n$-disk. We will first show that $f$ is a weak equivalence.

Denote by $I$ the unit interval $\left[0,1\right],$ and let $E\left(f\right)$ denote the homotopy fiber of $f,$
$$\xymatrix{E\left(f\right) \ar[r] \ar[d] & B^I \ar[d]^-{ev(0)}\\
X \ar[r]^-{f} & B,}$$ where $ev(0)$ is evaluation at $0$.

We may factor $X \to B$ as $$X \to E\left(f\right) \stackrel{ev(1)}{-\!\!\!-\!\!\!-\!\!\!-\!\!\!-\!\!\!\longrightarrow} B$$ where $X \to E\left(f\right)$ is a homotopy equivalence and the evaluation map (at $1$) is a fibration. From the long exact sequence of homotopy groups resulting from the fiber sequence $$E\left(f\right)_b \to E_f \to B,$$ it suffices to show that for each $b \in B$ in the image of $f,$ the homotopy fiber $E\left(f\right)_b=ev(1)^{-1}\left(b\right)$ is weakly contractible.

Suppose we are given a based map $l:S^{n-1} \to E\left(f\right)_b.$ Identifying $D^n$ with the cone on $S^{n-1},$ this is the same as giving a commutative diagram

$$\xymatrix{S^{n-1} \ar[d]_-{l_1} \ar[r] & D^n \ar[d]^-{l_2}\\
X \ar[r]^-{f} & B,}$$
where $S^{n-1} \to D^n$ is the  canonical map.

Let $\tilde f$ denote the induced map
$$\xymatrix{X \times_B D^n \ar[d] \ar[r]^-{\tilde f} & D^n \ar[d]^-{l_2} \\
X \ar[r]^-{f} & B.}$$
By our hypothesis, $\tilde f$ is a weak homotopy equivalence. Moreover, it is easy to check that the map $$l:S^{n-1} \to E\left(f\right)_b$$ factors through the canonical map $$q:E(\tilde f)_b \to E(f)_b,$$ say $l=ql'$ for some $$l':S^{n-1} \to E(\tilde f)_b.$$  As $\tilde f$ is a weak equivalence, $E(\tilde f)_b$ is weakly contractible. So, $l'$ is null-homotopic, and hence so is $l.$ It follows that $E\left(f\right)_b$ is also weakly contractible, and thus $f$ is a weak equivalence.

Moreover, $f$ is in fact a universal weak equivalence since if $T \to X$ is any map, the induced map $X \times_B T$ satisfies the same hypothesis that $f$ does.
\end{proof}

\begin{cor}\label{cor:tgood}
Every quasi-shrinkable map is a universal weak equivalence.
\end{cor}

\begin{dfn}
A representable map $\X \to \Y$ of stacks on $\cgh$ (with respect to either the $\cgt$-topology or the open cover topology) is said to be shrinkable, locally-shrinkable, $\cgt$-locally shrinkable, quasi-shrinkable, or a universal weak equivalence, if and only if for every map $T \to \X$ from a topological space, the induced map $\X \times_{\Y} T \to T$ is.
\end{dfn}

\begin{lem}\label{lem:cgtshr}
Let $f:\X \to \Y$ and $g:\Y' \to \Y$ be morphisms in $\St_{\cgt}\left(\cgh\right)$ such that $g$ is a $\cgt$-epimorphism and the induced map $$\X \times_{\Y} \Y'  \to \Y'$$ is a representable $\cgt$-locally shrinkable map. Then $f$ is also representable and $\cgt$-locally shrinkable.
\end{lem}
\begin{proof}
Let $h:T \to \Y$ be arbitrary. Choose a $\cgt$-cover $\left(V_\alpha\right)$ of $T$ such that for all $\alpha,$ there is a $2$-commutative diagram
$$\xymatrix{V_\alpha \ar[r] \ar[d]_-{h_\alpha} & T \ar[d]^-{h}\\
\Y' \ar[r]^-{g} & \Y.}$$
Note, by assumption, the induced maps $$\X \times_{\Y} V_\alpha \to V_\alpha$$ are $\cgt$-locally shrinkable maps of topological spaces. By refining this $\cgt$-cover if necessary, we can arrange that each of these maps is in fact shrinkable. It follows that $\X \times_{\Y} T$ is a topological space, and that the collection $\left(\X \times_{\Y} V_\alpha\right)$ is a $\cgt$-cover of it. Since the restriction of $$\X \times_{\Y} T \to T$$ to each element of this cover is $$\X \times_{\Y} V_\alpha \to V_\alpha,$$ it follows that $$\X \times_{\Y} T \to T$$ is $\cgt$-locally shrinkable.
\end{proof}

\begin{thm}\label{thm:atlslcsh}
Let $\X \simeq \left[\G\right]_\cgt$ be a compactly generated stack. Then the atlas $\left\|\G\right\| \to \X$ is $\cgt$-locally shrinkable.
\end{thm}

\begin{proof}
In \cite{NoohiH}, it is shown that we have a $2$-Cartesian diagram

$$\xymatrix{ \left\|\Ee\G\right\| \ar[r]^{f} \ar[d] & \G_0 \ar[d]\\
\left\|\G\right\| \ar[r]^{\varphi} & \left[\G\right]\\}$$
with $f$ shrinkable. Now, the stackification $2$-functor $a_\cgt$ commutes with finite weak limits, hence, the following is also a $2$-Cartesian diagram:

$$\xymatrix{ \left\|\Ee\G\right\| \ar[r]^{f} \ar[d] & \G_0 \ar[d]\\
\left\|\G\right\| \ar[r]^{\bar \varphi} & \X.\\}$$
Since the map $\G_0 \to \X$ is a $\cgt$-epimorphism and $f$ is shrinkable, it follows from Lemma \ref{lem:cgtshr} that $\bar \varphi$ is $\cgt$-locally shrinkable.
\end{proof}

\begin{cor}\label{cor:uniwkatls}
Let $\X \simeq \left[\G\right]_\cgt$ be a compactly generated stack. Then the atlas $\left\|\G\right\| \to \X$ is a universal weak equivalence.
\end{cor}

\begin{proof}
This follows immediately from Lemma \ref{lem:cgqsh} and Corollary \ref{cor:tgood}.
\end{proof}

\begin{cor}\label{cor:Morinvhtpy}
Let $\phi:\G \to \h$ be a $\cgt$-Morita equivalence between two topological groupoids $\G$ and $\h$. Then $\phi$ induces a weak homotopy equivalence

$$\left\|\phi\right\|: \left\|\G\right\| \to \left\|\h\right\|.$$
\end{cor}

\begin{proof}

Let $\X:=\left[\G\right]_\cgt \simeq \left[\h\right]_\cgt$. Then each atlas $\left\|\G\right\| \to \X$ and $\left\|\h\right\| \to \X$ is a universal weak equivalence. The following diagram $2$-commutes (with the outer square Cartesian):

$$\xymatrix{ \left\|\G\right\| \times_\X \left\|\h\right\| \ar[r]^-{\beta} \ar[d]_{\alpha} & \left\|\h\right\| \ar[d]\\
\left\|\G\right\| \ar[r] \ar@{.>}[ur]_{\left\|\varphi\right\|} & \X.\\}$$
Since each atlas is a universal weak equivalence, $\alpha$ and $\beta$ are weak equivalences, and hence so is $\left\|\varphi\right\|$.
\end{proof}

\begin{ex}
Let $X$ be a topological space. Consider the inclusion of all its compact subsets $\left(K_\alpha \hookrightarrow X\right)$. This is a $\cgt$-cover, so the associated groupoid $$\mathcal{K}=\left(\coprod K_\alpha \cap K_\beta \rightrightarrows \coprod K_\alpha\right)$$ is $\cgt$-Morita equivalent to $X$. It follows from Corollary \ref{cor:Morinvhtpy} that $\left\|\K\right\|$ is weakly homotopy equivalent to $X$.
\end{ex}

We now copy Noohi in \cite{NoohiH} to give a functorial assignment to each compactly generated stack a weak homotopy type.

Given a bicategory $\C,$ denote the $1$-category obtained by identifying equivalent $1$-morphisms by $\tau_1\left(\C\right).$ Suppose we are given a full sub-bicategory $B$ of $\C$ which is in fact (equivalent to) a $1$-category, and is closed under pullbacks. For example, consider $\C$ to be the bicategory of compactly generated stacks, $\Ctst,$ and for $B$ to the category of compactly generated Hausdorff spaces $\cgh.$ Let $R$ be a class of morphisms in $B$ which contains all isomorphisms, and is stable under pullback. Let $\tilde R$ denote the class of morphisms $f:\X \to \Y$ in $\C$ such that for every $h:T \to \Y,$ with $T$ in $B,$ the weak pullback $\X \times_{\Y} T$ is in $B$ and the induced map $$\X \times_{\Y} T \to T$$ is in $R.$ In the case where $\C=\Ctst$ and $B=\cgh,$ then $f$ is a representable map with property $R$.

\begin{lem}\cite{NoohiH}\label{lem:Noohitheta}
In the set up just described, if for every object $\X$ of $\C$ there exists a morphism $$\varphi\left(\X\right):\Theta\left(\X\right) \to \X$$ in $\tilde R$ from an object $\Theta\left(\X\right)$ of $B,$ then there is an induced adjunction
$$\xymatrix{\tilde R^{-1}\tau_1\left(\C\right) \ar@<-.5ex>_-\Theta[r] & R^{-1}B \ar@<-.5ex>[l]_-y}$$
with $y \mspace{2mu} \rt \mspace{2mu} \Theta,$ and with $y$ fully faithful. Moreover, the components of the co-unit of this adjunction are in $\tilde R.$
\end{lem}

\begin{thm}\label{thm:functorhomtype}
There exists a functor $\Omega:\Ctst \to \operatorname{Ho}\left(\T\right)$ assigning to each compactly generated stack $\X,$ a weak homotopy type. Moreover, for each $\X,$ there is a $\cgt$-atlas $X \to \X,$ which is a universal weak equivalence from a space $X$ whose homotopy type is $\Omega\left(\X\right).$
\end{thm}
\begin{proof}
In the previous lemma, let $\C=\Ctst,$ $B=\cgh,$ and let $R$ be the class of universal weak equivalences. Use Corollary \ref{cor:uniwkatls} to pick for each compactly generated stack $\X,$ a $\cgt$-atlas which is a universal weak equivalence. Lemma \ref{lem:Noohitheta} implies that there is an induced adjunction
$$\xymatrix{\tilde R^{-1}\tau_1\left(\Ctst\right) \ar@<-.5ex>_-\Theta[r] & R^{-1}\cgh \ar@<-.5ex>[l]_-y},$$ where the unit of this adjunction is an equivalence, and the co-unit for $\X,$ $$y\Theta\left(\X\right) \to \X,$$ is the chosen atlas.

Let $S$ denote the class of weak homotopy equivalences in $\cgh.$ Let $T$ denote $y\left(S\right).$ Then, since $$y\left(S\right)=T,$$ and $$\Theta\left(T\right)=S,$$ it follows that there is an induced adjunction
$$\xymatrix{S^{-1} \left(\tilde R^{-1}\tau_1\left(\Ctst\right)\right) \ar@<-.5ex>_-{\bar\Theta}[r] & S^{-1}\left(R^{-1}\cgh\right) \ar@<-.5ex>[l]_-{\bar y}}.$$
Note that there are canonical equivalences $$S^{-1} \left(\tilde R^{-1}\tau_1\left(\Ctst\right)\right) \simeq S^{-1}\tau_1\left(\Ctst\right),$$ and $$S^{-1}\left(R^{-1}\cgh\right) \simeq S^{-1}\cgh.$$ Moreover, since every topological space has the weak homotopy type of a compactly generated Hausdorff space, $S^{-1}\cgh$ is equivalent to the homotopy category of spaces, $\operatorname{Ho}\left(\T\right).$ Note that the following diagram $2$-commutes
$$\xymatrix{& \tilde R^{-1}\tau_1\left(\Ctst\right) \ar[r]^-{\Theta} \ar[d] & R^{-1}\cgh \ar[d]\\
\Ctst \ar[ru] \ar[r] & S^{-1}\tau_1\left(\Ctst\right) \ar[r]^-{\bar \Theta} & \operatorname{Ho}\left(\T\right).}$$
Denote either (naturally isomorphic) composite by $\Omega:\Ctst \to \operatorname{Ho}\left(\T\right).$
\end{proof}

\subsection{Comparison with Topological Stacks}\label{sec:comp}

In this subsection, we will extend the results of the previous section to give a functorial assignment of a weak homotopy type to a wider class of stacks, which include all compactly generated stacks and all topological stacks. We will then show that for a given topological stack $\X,$ the induced map to its associated compactly generated stack $a_{\cgt}\left(\X\right)$ is a weak homotopy equivalence. Finally, we will show in what sense compactly generated stacks are to topological stacks what compactly generated spaces are to topological spaces (Theorem \ref{thm:kellyref}).

\begin{prop}\label{prop:qausitopol}
For a stack $\X$ over compactly generated Hausdorff spaces (with respect to the open cover topology) whose diagonal  $$\Delta:\X \to \X \times \X$$ is representable, the following conditions are equivalent:
\begin{itemize}
\item[i)] the $\cgt$-stackification of $\X$ is a compactly generated stack,
\item[ii)] there exists a topological space $X$ and a morphism $X \to \X$ such that, for all spaces $T$, the induced map $$T \times_{\X} X \to T$$ admits local sections with respect to the topology $\cgt$ (i.e. is a $\cgt$-covering morphism; See Definition \ref{dfn:covmor}),
\item[iii)] there exists a topological space $X$ and a morphism $X \to \X$ such that, for all compact Hausdorff spaces $T$, the induced map $$T \times_{\X} X \to T$$ admits local sections (i.e. is an epimorphism),
\item[iv)] there exists a topological space $X$ and a morphism $X \to \X$ such that, for all locally compact Hausdorff spaces $T$, the induced map $$T \times_{\X} X \to T$$ admits local sections (i.e. is an epimorphism),
\item[v)] there exists a topological stack $\bar \X$ and a map $q:\bar \X \to \X$ such that, for all locally compact Hausdorff spaces $T,$ $$q\left(T\right):\bar \X\left(T\right) \to \X\left(T\right)$$ is an equivalence of groupoids.
\end{itemize}
\end{prop}

\begin{proof}
Suppose that condition $i)$ is satisfied. Note that condition $ii)$ is equivalent to saying that there exists a $\cgt$-covering morphism $X \to \X$ from a topological space (See Definition \ref{dfn:covmor}.) Let $$p:X \to a_{\cgt}\left(\X\right)$$ be a locally compact Hausdorff atlas for the compactly generated stack  $a_{\cgt}\left(\X\right).$ Then it factors (up to equivalence) as $$X \stackrel{x}{\longrightarrow} \X \to a_{\cgt}\left(\X\right),$$ for some map $$x:X \to \X.$$ Note that the $\cgt$-stackification of $x$ is (equivalent to) $p,$ hence is an epimorphism in $\St_{\cgt}\left(\cgh\right).$ This implies $x$ is a $\cgt$-covering morphism. So $i) \Rightarrow ii).$

Since any $\cgt$-cover of a locally compact Hausdorff space can be refined by an open one, $ii) \left(\Rightarrow iv)\right) \Rightarrow iii).$

From Corollary \ref{cor:eql2tps}, it follows that there is an equivalence of bicategories $$\Lambda:\St\left(\ch\right) \to \St_{\cgt}\left(\cgh\right),$$ such that for every stack $\X$ on $\cgh$ with respect to the open cover topology, $$\Lambda\left(j^*\X\right) \simeq a_{\cgt}\left(\X\right).$$
Hence $iii) \left(\Rightarrow ii)\right) \Rightarrow iv).$

Suppose $iv)$ holds. Then $iv) \Rightarrow iii)$ trivially, and $iii) \Rightarrow ii)$ by the above argument. So there exists a $\cgt$-covering map $X \to \X.$ Hence, the induced map $X \to a_{\cgt}\left(\X\right)$ is an epimorphism (in particular, this implies $i)).$ Consider the induced map $$\alpha:\left[X \times_{\X} X \rightrightarrows X\right] \to \X.$$ It is a monomorphism, and since $X \to \X$ is a $\cgt$-covering morphism, $\alpha$ is too. Since stackification is left-exact, it preserves monomorphisms and hence $a_{\cgt}\left(\alpha\right)$ is an equivalence between the compactly generated stack $$\left[X \times_{\X} X \rightrightarrows X\right]_{\cgt}$$ and $a_{\cgt}\left(\X\right).$ Proposition \ref{prop:eqparat} implies that $\alpha$ satisfies $v).$ Hence, $iv) \Rightarrow v).$

Suppose that $v)$ holds for a morphism $q:\bar \X \to \X$ from a topological stack. Then $$j^*\left(\bar\X\right) \to j^*\X$$ is an equivalence. Hence, $a_{\cgt}\left(q\right)$ is an equivalence between $a_{\cgt}\left(\X\right)$ and the compactly generated stack $a_{\cgt}\left(\bar\X\right).$ Hence $v) \Rightarrow i).$
\end{proof}

\begin{dfn}\index{quasi-topological stack}
A stack $\X$ over compactly generated Hausdorff spaces (with respect to the open cover topology) whose diagonal  $$\Delta:\X \to \X \times \X$$ is representable, is \textbf{quasi-topological} if any of the equivalent conditions of Proposition \ref{prop:qausitopol} hold. Denote the full sub-bicategory of $\St\left(\cgh\right)$ consisting of the quasi-topological stacks by $\QCtst.$
\end{dfn}

The following proposition is immediate.
\begin{prop}
If $\X$ is a stack over compactly generated Hausdorff spaces which is topological, paratopological, or compactly generated, then it is quasi-topological.
\end{prop}



\begin{lem}\label{lem:pbklk}
Let $\X$ be a quasi-topological stack, and let $$h:T \to a_{\cgt}\left(\X\right)$$ be a map from a locally compact Hausdorff space. Then $T \times_{a_{\cgt}\left(\X\right)} \X$ is a topological space and the induced map $$T \times_{a_{\cgt}\left(\X\right)} \X \to T$$ is a homeomorphism.
\end{lem}
\begin{proof}
Let $X \to a_{\cgt}\left(\X\right)$ be a locally compact Hausdorff $\cgt$-atlas. Then it factors (up to equivalence) as $$X \stackrel{x}{\longrightarrow} \X \stackrel{\varsigma_{\X}}{\longrightarrow} a_{\cgt}\left(\X\right).$$
Moreover, as $T$ is locally compact Hausdorff, there  is a $2$-commutative lift $h'$
$$\xymatrix{ & \X \ar[d]^-{\varsigma_\X}\\T \ar[r]_-{h} \ar@{..>}[ru]^-{h'} & a_{\cgt}\left(\X\right).}$$
Consider the weak  pullback
$$\xymatrix{P \ar[d] \ar[r] & X \ar[d]^-{x}\\
T \ar[r]^-{h'} & \X.}$$
As a sheaf, $P$ assigns each space $Z,$ the set of triples $\left(f,g,\alpha\right)$ with $$f:Z \to T,$$ $$g:Z \to X,$$ and $$\alpha:xf \Rightarrow h'g.$$ Since the diagonal of $\X$ is representable, $P$ is represented by a compactly generated Hausdorff space.

Consider now the weak pullback  diagram
$$\xymatrix{P' \ar[d] \ar[r] & X \ar[d]^-{\varsigma_{\X}\circ x}\\
T \ar[r]^-{\varsigma_{\X} \circ h'} & a_{\cgt}\left(\X\right).}$$
The sheaf $P'$ is again representable, and it assigns each space $Z$ the set of triples $\left(f,g,\beta\right)$ with $$f:Z \to T,$$ $$g:Z \to X,$$ and $$\beta:\varsigma_{\X}\circ xf \Rightarrow \varsigma_{\X}\circ h'g.$$
Consider the induced map $P \to P'$ given by composition with $\varsigma_{\X}.$ Since for every locally compact Hausdorff space $S,$ $\varsigma_{\X}\left(S\right)$ is an equivalence of groupoids, it follows that the induced map $y_{\ch}\left(P\right) \to y_{\ch}\left(P'\right)$ is an isomorphism, where $$\yk:\cg \to \Skp$$
is the functor which assigns a compactly generated Hausdorff space $X$ the presheaf $S \mapsto \Hom_{\cg}\left(S,X\right).$ Since $y_{\ch}$ is fully-faithful, it follows that the induced map $P \to P'$ is an isomorphism.

Finally, regard the following $2$-commutative diagram:
$$\xymatrix{P' \ar[r] \ar[d] & X \ar[d]\\
T \ar[r]^-{h'} \ar[d]_{id_T} & \X \ar[d]^-{\varsigma_{\X}}\\
T \ar[r]^-{h} & a_{\cgt}\left(\X\right).}$$
The outer square is Cartesian, and so is the upper-square. It follows that $$\xymatrix{T \ar[r]^-{h'} \ar[d]_{id_T} & \X \ar[d]^-{\varsigma_{\X}}\\
T \ar[r]^-{h} & a_{\cgt}\left(\X\right)}$$ is Cartesian as well.
\end{proof}

\begin{cor}\label{cor:funhty}
For every quasi-topological stack $\X,$ there exists a representable universal weak equivalence $$\Theta\left(\X\right) \to \X,$$ from a topological space $\Theta\left(\X\right).$
\end{cor}

\begin{proof}
Let $\X$ be a quasi-topological stack, and let $X \to a_{\cgt}\left(\X\right)$ be a locally compact Hausdorff $\cgt$-atlas. Then it factors (up to equivalence) as $$X \stackrel{x}{\longrightarrow} \X \stackrel{\varsigma_{\X}}{\longrightarrow} a_{\cgt}\left(\X\right).$$ Denote by $\G$ the topological groupoid $$X \times_{\X} X \rightrightarrows X.$$ There is a canonical map $\left[\G\right] \to \X$ and the unit map $\varsigma_{\left[\G\right]}$ factors as $$\left[\G\right] \to \X \stackrel{\varsigma_{\X}}{\longrightarrow} a_{\cgt}\left(\X\right).$$ The composite $$\left\|\G\right\| \to \left[\G\right] \to \X \stackrel{\varsigma_{\X}}{\longrightarrow} a_{\cgt}\left(\X\right)$$ is a representable quasi-shrinkable morphism, by Theorem \ref{thm:atlslcsh}. From Lemma \ref{lem:pbklk}, it follows that $$\left\|\G\right\| \to \left[\G\right] \to \X$$ is a representable quasi-shrinkable morphism as well, and in particular, a representable universal weak equivalence, by Corollary \ref{cor:tgood}.
\end{proof}

\begin{thm}
There exists a functor $\Omega:\QCtst \to \operatorname{Ho}\left(\T\right)$ assigning to each quasi-topological stack $\X,$ a weak homotopy type. Moreover, for each $\X,$ there is a representable universal weak equivalence $$X \to \X,$$ from a space $X$ whose homotopy type is $\Omega\left(\X\right).$ Furthermore, we can arrange for the functor $\Omega$ to restrict to the one of Theorem \ref{thm:functorhomtype} on compactly generated stacks.
\end{thm}

\begin{proof}
Using the notation of Lemma \ref{lem:Noohitheta}, let $\C=\QCtst,$\\ $B=\cgh,$ and let $R$ be the class of universal weak equivalences. Using the notation of the proof of Corollary \ref{cor:funhty}, for each quasi-topological stack $\X,$ let $$\varphi\left(\X\right):\Theta\left(\X\right)=\left\|\G\right\| \to \left[\G\right] \to \X.$$ The rest is identical to the proof of Theorem \ref{thm:functorhomtype}.
\end{proof}

\begin{rmk}
This agrees with the functorial construction of a weak homotopy type of a topological or paratopological stack given in \cite{NoohiH}, by construction.
\end{rmk}

\begin{dfn}
A morphism $f:\X \to \Y$ in $\QCtst$ is a \textbf{weak homotopy equivalence} if $\Omega\left(f\right)$ is an isomorphism.
\end{dfn}

\begin{thm}
Let $\X$ be a quasi-topological stack. Then the unit map $$\varsigma_{\X}:\X \to a_\cgt \left(\X\right)$$ induces an equivalence of groupoids
$$\X\left(Y\right) \to a_\cgt\X\left(Y\right)$$
for all locally compact Hausdorff spaces $Y$, and $\varsigma_\X$ is a weak homotopy equivalence.
\end{thm}

\begin{proof}
The first statement follows immediately from Proposition \ref{prop:eqparat}. For the second, letting $R$ denote the class of universal weak equivalences, we can factor $\Omega$ as
$$\QCtst \to R^{-1} \QCtst \stackrel{\Theta}{\longrightarrow} R^{-1}\cgh \to \operatorname{Ho}\left(\T\right).$$ To show $\Omega\left(\varsigma_\X\right)$ is an isomorphism, it suffices to show $\Theta\left(\varsigma_\X\right)$ is. From \cite{NoohiH}, an arrow between two spaces $X$ and $Y$ in $R^{-1}\cgh$ is a span $\left(r,g\right)$ of the form
$$\xymatrix{ & T \ar[ld]_-{r} \ar[rd]^-{g} & \\
X & & Y,}$$ with $r$ a universal weak equivalence. Moreover, if $f:\X \to \Y$ is a morphism of quasi-topological stacks, $\Theta\left(f\right)$ is given by the span $\left(w,f'\right)$ provided by the diagram
$$\xymatrix{\Theta\left(\X\right) \times_{\Y} \Theta\left(\Y\right) \ar[d]_-{w} \ar[rr]^-{f'} & & \Theta\left(\Y\right) \ar[d]^-{\varphi\left(\Y\right)} \\
\Theta\left(\X\right) \ar[r]^-{\varphi\left(\X\right)} & \X \ar[r]^-{f} & \Y.}$$
Such a span is an isomorphism if and only if $f'$ is a universal weak equivalence. It follows that $\Theta\left(\varsigma_\X\right)$ is given by the span defined by the diagram
$$\xymatrix{\Theta\left(\X\right) \times_{\Y} \Theta\left(a_\cgt \left(\X\right)\right) \ar[d]_-{w} \ar[rr]^-{f'} & & \Theta\left(a_\cgt \left(\X\right)\right) \ar[d]^-{\varphi\left(a_\cgt \left(\X\right)\right)} \\
\Theta\left(\X\right) \ar[r]^-{\varphi\left(\X\right)} & \X \ar[r]^-{\varsigma_\X} & a_\cgt \left(\X\right).}$$
Notice that $$\Theta\left(\X\right) \stackrel{\varphi\left(\X\right)}{-\!\!\!-\!\!\!-\!\!\!-\!\!\!-\!\!\!\longrightarrow} \X \stackrel{\varsigma_\X}{-\!\!\!-\!\!\!-\!\!\!-\!\!\!-\!\!\!\longrightarrow} a_\cgt \left(\X\right)$$ is a representable universal weak equivalence. It follows that $f'$ is as well, so we are done.
\end{proof}

\begin{cor}\label{prop:comp1}
Let $\X$ be a topological or paratopological stack. Then the unit map $$\varsigma_{\X}:\X \to a_\cgt \left(\X\right)$$ induces an equivalence of groupoids
$$\X\left(Y\right) \to a_\cgt\X\left(Y\right)$$
for all locally compact Hausdorff spaces $Y$. Moreover, $\varsigma_\X$ is a weak homotopy equivalence.
\end{cor}

In particular, to any topological stack, there is a canonically associated compactly generated stack of the same weak homotopy type which restricts to the same stack over locally compact Hausdorff spaces. Conversely, if $$\Y \simeq \left[\h\right]_\cgt$$ is a compactly generated stack, $\left[\h\right]$ is an associated topological stack for which the same is true.

\begin{thm}\index{mapping stack}
Let $\X$ and $\Y$ be topological stacks such that $\Y$ admits a locally compact Hausdorff atlas. Then $\Map\left(\Y,\X\right)$ is a paratopological stack, $\Map\left(a_\cgt\left(\Y\right),a_\cgt\left(\X\right)\right)$ is a compactly generated stack, and there is a canonical weak homotopy equivalence
$$\Map\left(\Y,\X \right) \to \Map\left(a_\cgt\left(\Y\right),a\left(\X\right)\right).$$
Moreover, $\Map\left(\Y,\X \right)$ and $\Map\left(a_\cgt\left(\Y\right),a_\cgt\left(\X\right)\right)$ restrict to the same stack over locally compact Hausdorff spaces.
\label{thm:hommap}
\end{thm}

\begin{proof}
The fact that $\Map\left(\Y,\X\right)$ is a paratopological stack follows from Theorem \ref{thm:cw}, and that $\Map\left(a_\cgt\left(\Y\right),a_\cgt\left(\X\right)\right)$ is a compactly generated stack follows from Theorem \ref{thm:carclos}. To prove the rest, it suffices to prove that $$a_{\cgt}\left(\Map\left(\X,\Y\right)\right)\simeq \Map\left(a_\cgt\left(\X\right),a_{\cgt}\left(\Y\right)\right).$$ For this, it is  enough to show that they restrict to the same stack over compact Hausdorff spaces. Let $T$ be a compact Hausdorff space. Then $$a_{\cgt}\left(\Map\left(\X,\Y\right)\right)$$ assigns $T$ the groupoid $$\Hom\left(T \times \X,\Y\right),$$ since it agrees with $\Map\left(\X,\Y\right)$ along locally compact Hausdorff spaces. From Corollary \ref{cor:tiredxx}, since $T \times \X$ admits a locally compact atlas, this is in turn equivalent to the groupoid $$\Hom\left(a_{\cgt}\left(T \times \X\right),a_{\cgt}\left(\Y\right)\right) \simeq \Map\left(a_\cgt\left(\X\right),a_{\cgt}\left(\Y\right)\right)\left(T\right).$$
\end{proof}

We end this paper by showing exactly in what sense compactly generated stacks are to topological stacks what compactly generated spaces are to topological spaces:

Recall that there is an adjunction

$$\xymatrix@C=1.5cm{\cgr \ar@<-0.5ex>[r]_-{v} & \T \ar@<-0.5ex>[l]_-{k},}$$
exhibiting compactly generated spaces as a co-reflective subcategory of the category of topological spaces, and for any space $X,$ the co-reflector $$vk\left(X\right) \to X$$ is a weak homotopy equivalence.

We now present the $2$-categorical analogue of this statement:

\begin{thm}\label{thm:kellyref}
There is a $2$-adjunction
$$\xymatrix@C=1.5cm{\Ctst \ar@<-0.5ex>[r]_-{v} & \mathfrak{\Tst} \ar@<-0.5ex>[l]_-{k},}$$
$$v \rt \mspace{2mu} k,$$
exhibiting compactly generated stacks as a co-reflective sub-bicategory of topological stacks, and for any topological stack $\X$, the co-reflector  $$vk\left(\X\right) \to \X$$ is a weak homotopy equivalence. A topological stack is in the essential image of the $2$-functor $$v:\Ctst \to \Tst$$ if and only if it admits a locally compact Hausdorff atlas.
\end{thm}
\begin{proof}
Let us first start with the $2$-functor $$k:\Tst \to \Ctst.$$ We define it to be the restriction of the stackification $2$-functor $$a_{\cgt}:Gpd^{\cgh^{op}} \to St_{\cgt}\left(\cgh\right)$$ to $\Tst.$ Note that since every open cover is a $\cgt$-cover, for all topological groupoids $\G$ and $\h,$ there is a canonical full and faithful functor $$Bun_\h\left(\G\right) \to Bun^{\cgt}_\h\left(\G\right).$$ In fact, it is literally an inclusion on the level of objects. These assemble into a homomorphism of bicategories $$\tilde k:Bun\cgh Gpd \to Bun^{\cgt}\cgh Gpd$$ which is the identify on objects, $1$-morphisms, and $2$-morphisms (but it is not $2$-categorically full and faithful). Composing with the canonical equivalences,

$$\Tst \stackrel{\sim}{\longrightarrow} Bun\cgh Gpd \stackrel{\tilde k}{\longrightarrow} Bun^{\cgt}\cgh Gpd \stackrel{\sim}{\longrightarrow} \Ctst$$
is a factorization (up to equivalence) of $a_{\cgt}|_{\Tst}.$
We will construct a left $2$-adjoint to $\tilde k$. For all topological groupoids $\G,$ let $\tilde v\left(\G\right)$ be the \v{C}ech groupoid $\G_{\K}$ of $\G$ with respect to the $\cgt$-cover of $\G_0$ by all its compact subsets. In particular, $\tilde v\left(\G\right)$ has a locally compact Hausdorff object space. The canonical map $\G_{\K} \to \G$ is a $\cgt$-Morita equivalence. Denote the associated $\cgt$-principal $\G$-bundle over $\G_{\K}$ by ${1^k}_\G$. Since $\left(\G_{\K}\right)_0$ is locally compact Hausdorff, ${1^k}_\G$ is an ordinary principal bundle $$1^{k}_{\G} \in Bun_\G\left(\tilde v\left(\G\right)\right)_0.$$ Since $\G_{\K} \to \G$ is a $\cgt$-Morita equivalence, there exists a $\cgt$-principal $\G_\K$-bundle over $\G,$ $r_\G$ and isomorphisms
\begin{eqnarray*}
\alpha_\G:r_\G \otimes {1^k}_\G &\stackrel{\sim}{\longrightarrow}& id_{\G}\\
\beta_\G:id_{\G_\K} &\stackrel{\sim}{\longrightarrow}& {1^k}_\G \otimes r_\G,\\
\end{eqnarray*}
where the tensor product symbol $\otimes$ denotes composition in the bicategory of principal bundles.
The assignment $\G \mapsto \tilde v\left(\G\right)=\G_\K$ extends to a homomorphism of bicategories $$\tilde v:Bun^{\cgt}\cgh Gpd \stackrel{k}{\longrightarrow} Bun\cgh Gpd,$$ by
\begin{eqnarray*}
\tilde v_{\h,\G}:Bun^{\cgt}_\G\left(\h\right)_0 &\to& Bun_{\tilde v\left(\G\right)}\left(\tilde v\left(\h\right)\right)\\
P &\mapsto& r_\h \otimes P \otimes {1^k}_\G,\\
\end{eqnarray*}
and similarly on $2$-cells.
Note that for $\G$ and $\h$ topological groupoids, there is a natural equivalence of groupoids
\begin{eqnarray*}
\Hom\left(\tilde v\left(\G\right),\h\right) &=& Bun_{\h}\left(\G_\K\right)\\
&=& Bun^{\cgt}_\h\left(\G_\K\right) \mbox{(since $\G_\K$ has locally compact Hausdorff object space)}\\
&\simeq& Bun^{\cgt}_\h\left(\G\right) \mbox{(since $\G_\K$ is $\cgt$-Morita equivalent to $\G$)}\\
&=& \Hom\left(\G,k\left(\h\right)\right),
\end{eqnarray*}
which sends a principal $\G$-bundle $P$ over $\h_\K$ to $P \otimes r_\G.$
An inverse for this equivalence is given by sending $$P \in Bun^{\cgt}_\G\left(\h\right)$$ to $P \otimes 1^k_\h.$
These equivalences define an adjunction of bicategories, $\tilde v \rt \mspace{2mu} \tilde k.$ The unit of this adjunction is given by $$r_\G:\G \to \G_\K=\tilde k \tilde v\left(\G\right) \mbox{ in $Bun^{\cgt}\cgh Gpd$,}$$ which is an equivalence. It follows that $v$ is bicategorically full and faithful. The co-unit is given by $${1^k}_\G:\tilde v \tilde k\left(\G\right)=\G_\K \to \G  \mbox{ in $Bun\cgh Gpd$.}$$
The induced adjunction $$\xymatrix@C=1.5cm{\Ctst \simeq Bun^{\cgt}\cgh Gpd \ar@<-0.5ex>[r]_-{\tilde v} & GpdBun\cgh\simeq \mathfrak{\Tst}\ar@<-0.5ex>[l]_-{\tilde k},}$$ is the desired $v \rt k.$
The $2$-functor $v$ sends a compactly generated stack equivalent to $\left[\G\right]_{\cgt}$ to $\left[\G_\K\right],$ which has a locally compact Hausdorff atlas. From the general theory of adjunctions, the essential image is precisely the sub-bicategory of topological stacks for over which $k=a_{\cgt}$ restricts to a full and faithful $2$-functor. We have already shown that the essential image is contained in those topological stacks which admit a locally compact Hausdorff atlas. By Corollary \ref{cor:fflk}, $a_{\cgt}$ restricted to this sub-bicategory is full and faithful, hence the essential image of $v$ is topological stacks which admit a locally compact Hausdorff atlas.

It remains to show that the co-unit is a weak homotopy equivalence. Let $\left[\G\right]$ be a topological stack. Then the co-unit is given by the canonical map $$\varepsilon_{\left[\G\right]}:\left[\G_K\right] \to \left[\G\right].$$ Notice that the following diagram is $2$-commutative:
$$\xymatrix{\left[\G_\K\right] \ar[r]^-{\varepsilon_{\left[\G\right]}}
\ar[rd]_-{\varsigma_{\left[\G_\K\right]}} & \left[\G\right] \ar[d]^-{\varsigma_{\left[\G\right]}} \\
& a_{\cgt}\left(\left[\G\right]\right).}$$
By Corollary \ref{prop:comp1}, the maps $\varsigma_{\left[\G\right]}$ and $\varsigma_{\left[\G_\K\right]}$ are weak homotopy equivalences. It follows that so is $\varepsilon_{\left[\G\right]}$.
\end{proof}

\begin{alphasection}
\section{Stacks}
\label{A:1}

Let $\left(\C,J\right)$ be a small Grothendieck site and let $Gpd$ denote the bicategory of groupoids\footnote{Technically, we mean those groupoids which are equivalent to a small category.}. Denote by $\pst$ the bicategory of weak presheaves $\C^{op} \to Gpd$. That is, weak contravariant $2$-functors from $\C$ to $Gpd$. We call this the bicategory of \textbf{weak presheaves} on $\C$. There exists a canonical inclusion $$\left(\mspace{2mu} \cdot \mspace{2mu}\right)^{id}:\ps \to \pst,$$ where, each presheaf $F$ is sent to the weak presheaf which assigns to each object $C$ the category $\left(F(C)\right)^{id}$ whose objects are $F\left(C\right)$ and whose arrows are all identities. If $C$ is an object of $\C$, we usually denote $\left(y_C\right)^{id}$ simply by $C$.

We also have a version of the Yoneda lemma:

\begin{lem}\cite{FGA}
\textbf{The $2$-Yoneda Lemma:}
If $C$ is an object of $\C$ and $\X$ a weak presheaf, then there is a natural equivalence of groupoids

$$\Hom_{\pst}\left(C,\X\right) \simeq \X\left(C\right).$$
\end{lem}

\begin{rmk}
The bicategory $\pst$ is both complete and cocomplete; weak limits are computed ``pointwise'':

$$\left(\hl \X_i\right)\left(X\right)= \hl \X_i\left(X\right)$$
where the weak limit to the right is computed in $Gpd$. Similarly for weak colimits.
\end{rmk}

\begin{dfn}\label{dfn:stack}
A weak presheaf $\X$ is called a \textbf{stack} if for every object $C$ and covering sieve $S$, the natural map

$$ \Hom_{\pst}\left(C,\X\right) \to \Hom_{\pst}\left(S,\X\right)$$
is an equivalence of groupoids.

If this map is fully faithful, $\X$ is called \textbf{separated} (or a \textbf{prestack}). Although it is not standard, if this map is faithful, we will call it \textbf{weakly separated}.

\end{dfn}
We denote the full sub-bicategory of $\pst$ consisting of those weak presheafs that are stacks by $St_J\left(\C\right)$.

It is immediate from the definition that the weak limit of any small diagram of stacks is again a stack.

If $\mathcal{B}$ is a basis for the topology $J$, then it suffices to check this condition for every sieve of the form $S_{\mathcal{U}}$, where $\mathcal{U}$ is a covering family. Namely, a weak presheaf is a stack if and only if for every covering family $\mathcal{U}=\left(U_i \to C\right)_i$ the induced map

\begin{equation*}
\X(C) \to \hl \left[ {\prod {\X\left(U_i\right)}} \rrarrow {\prod {\X\left(U_{ij}\right)}} \rrrarrow {\prod {\X\left(U_{ijk}\right)}}\right]
\end{equation*}
is an equivalence of groupoids.

If this map is fully faithful, $\X$ is is separated and if this map is faithful, it is weakly separated.

The associated groupoid

$$\mathfrak{Des}\left(\X,\mathcal{U}\right):= \hl \left[ {\prod {\X\left(U_i\right)}} \rrarrow {\prod {\X\left(U_{ij}\right)}} \rrrarrow {\prod {\X\left(U_{ijk}\right)}}\right],$$
obtained as weak limit of the above diagram of groupoids, is called the category of \textbf{descent data} for $\X$ at $\mathcal{U}$.

\begin{prop}
If $J$ is subcanonical and $\left(U_i \to C\right)_i$ is a covering family for an object $C$, then, \textbf{in the bicategory of stacks}

\begin{equation*}
C \simeq \hc \left[{\coprod{U_{ijk}}} \rrrarrow {\coprod{U_{ij}}} \rrarrow {\coprod{U_{i}}}\right]
\end{equation*}
\label{covlim}
\end{prop}

We will often simply write $$C \simeq \underset{U_i \rightarrow C} \hc U_i.$$

\begin{dfn}\label{dfn:repr}
A morphism $\varphi:\X \to \Y$ of stacks is called \textbf{representable} if for any object $C \in \C$ and any morphism $C \to \Y$, the weak pullback $C \times_{\Y} \X$ in the category of stacks is (equivalent to) an object $D$ of $\C$.
\end{dfn}

We can now define a $2$-functor $$\left(\mspace{2mu} \cdot \mspace{2mu}\right)^{+}:\pst \to \pst$$ by

\begin{equation*}
\X^{+}\left(C\right):=\underset{S\in Cov(C)} \hc \Hom_{\pst}\left(S,F\right).
\end{equation*}

We can alternatively define $\left(\mspace{2mu} \cdot \mspace{2mu}\right)^{+}$ by the equation

\begin{equation*}
\X^{+}\left(C\right):=\underset{\mathcal{U} \in cov(C)} \hc \mathfrak{Des}\left(\X,\mathcal{U}\right)
\end{equation*}

and obtain a naturally equivalent $2$-functor.

\begin{rmk}
The weak colimit in either definition must be indexed over a suitable bicategory of covers.
\end{rmk}

$\X$ is separated if and only if the canonical map $\X \to \X^+$ is a fully faithful, and weakly separated if and only if it is faithful. $\X$ is a stack if and only if this map is an equivalence.

If $\X$ is separated, then $\X^+$ is a stack, and if $\X$ is only weakly separated, then $\X^+$ is separated. Furthermore, $\X^+$ is always weakly separated, for any $\X$. Hence, $\X^{+++}$ is always a stack.

\begin{dfn}\label{dfn:stackification}
We denote by $a_J$ the $2$-functor $\X \mapsto \X^{+++}$. It is called the \textbf{stackification} $2$-functor \index{stackification}.
\end{dfn}
 If $\X$ is separated, $a_J\left(\X\right) \simeq \X^+$, and if $\X$ is a stack, $a_J\left(\X\right) \simeq \X$. The $2$-functor $a_J$ is left $2$-adjoint to the inclusion $$i:St_J\left(\C\right) \hookrightarrow \pst$$ and preserves finite weak limits.

\begin{rmk}
$St_J\left(\C\right)$ is both complete and cocomplete. Since $i$ is a right adjoint, it follows that the computation of weak limits of stacks can be done in the category $\pst$, hence can be done ``pointwise''. To compute the weak colimit of a diagram of stacks, one must first compute it in $\pst$ and then apply the stackification functor $a_J$.
\end{rmk}

We end by remarking that the bicategory $St_J\left(\C\right)$ is Cartesian closed. The exponent $\X^{\Y}$ of two stacks is given by

\begin{equation*}
\X^{\Y}\left(C\right)=\Hom_{\pst}\left(\Y \times C, \X\right),
\end{equation*}
and satisfies
\begin{equation*}
\Hom_{\pst}\left(\Z,\X^{\Y}\right) \simeq \Hom_{\pst}\left(\Y \times \Z, \X\right)
\end{equation*}
for all stacks $\Z$.

\label{B:1}

\section{Topological Groupoids and Topological Stacks}\label{sec:topgpd}
\subsection{Definitions and Some Examples}
This section is a review of topological groupoids and topological stacks. The material is entirely standard aside from the fact that some of the standard results are stated for more general Grothendieck topologies.

\begin{dfn}
A \textbf{topological groupoid} is a groupoid object in $\T,$ the category of topological spaces. Explicitly, it is a diagram

$$\xymatrix{ {\G_1 \times _{\G_0}\G_1} \ar[r]^(0.6){m} & \G_1 \ar@<+.7ex>^s[r]\ar@<-.7ex>_t[r] \ar@(ur,ul)^i & \G_0 \ar@/^1.65pc/^u [l] }$$
of topological spaces and continuous maps satisfying the usual axioms. Forgetting the topological structure (i.e. applying the forgetful functor from $\T$ to $\Set$), one obtains an ordinary (small) groupoid.
\end{dfn}

Topological groupoids form a bicategory with continuous functors as $1$-morphisms and continuous natural transformations as $2$-morphisms. We denote the bicategory of topological groupoids by $\T Gpd$.

\begin{rmk}
$\T Gpd$ has weak fibered products and arbitrary products. Since it is a $\left(2,1\right)$-category, by \cite{htt} it follows that it is a complete bicategory.
\end{rmk}


\begin{dfn}
Given a topological space $X$, we denote by $\left(X\right)^{id}$ the topological groupoid whose object and arrow space are both $X$ and all of whose structure maps are the identity morphism of $X$. The arrow space is the collection of all the identity arrows for the objects $X$.
\end{dfn}

\begin{dfn}
Given a topological space $X$, the \textbf{pair groupoid} $Pair\left(X\right)$ is the topological groupoid whose object space is $X$ and whose arrow space is $X \times X$, where an element $$\left(x,y\right) \in X \times X$$ is viewed as an arrow from $y$ to $x$ and composition is defined by the rule $$\left(x,y\right) \cdot \left(y,z\right)=\left(x,z\right).$$
\end{dfn}

\begin{dfn}
Given a continuous map $\phi:U \to X$, the \textbf{relative pair groupoid} $Pair\left(\phi\right)$ is defined to be the topological groupoid whose arrow space is the fibered product $U \times_{X} U$ and object space is $U$, where an element

$$\left(x,y\right) \in U \times_{X} U \subset U \times U$$ is viewed as an arrow from $y$ to $x$ and composition is defined by the rule $$\left(x,y\right) \cdot \left(y,z\right)=\left(x,z\right).$$  The pair groupoid of a space $X$ is the relative pair groupoid of the unique map from $X$ to the one-point space.
\end{dfn}

\begin{dfn}\label{dfn:cech}
Given a topological groupoid $\G$ and a continuous map $f:X \to \G_0$, there is an \textbf{induced groupoid} $f^*\left(\G\right)$, which is a topological groupoid whose object space is $X$ such that arrows between $x$ and $y$ in $f^*\left(\G\right)$ are in bijection with arrows between $f(x)$ and $f(y)$ in $\G$. When $X=\coprod\limits_\alpha  {U_\alpha  }$ with $\mathcal{U}=\left(U_\alpha \hookrightarrow X\right)_\alpha$ an open cover of $\G_0$ and $X \to \G_0$ the canonical map, $f^*\left(\G\right)$ is denoted by $\G_{\mathcal{U}}$. If in addition to this, $\G=\left(T\right)^{id}$ for some topological space $T$, then this is called the \textbf{\v{C}ech groupoid} associated to the cover $\mathcal{U}$ of $T$ and is denoted by $T_{\mathcal{U}}$.
\end{dfn}

\begin{rmk}
If the open cover $\mathcal{U}$ is instead a cover for a different Grothendieck topology on $\T$, the above still makes sense. This will be important later.
\end{rmk}

\begin{rmk}
A \v{C}ech groupoid is just a pair groupoid, for a map of the form $\coprod\limits_\alpha \to X.$
\end{rmk}

\subsection{Principal Bundles}

Principal bundles for topological groups, and more generally for topological groupoids, are classical objects of study. However, principal bundles (and many other objects involving a local triviality condition) should not be thought of as objects associated to the \emph{category} $\T$, but rather as objects associated to the \emph{Grothendieck site} $\left(\T,\mathcal{O}\right)$, where $\mathcal{O}$ is the \textbf{open cover topology} on $\T$. This Grothendieck topology is defined by declaring a family of maps $\left(O_\alpha \to X\right)_\alpha$ to be a covering family if and only if it constitutes an open cover of $X$. The concept of principal bundles generalizes to other Grothendieck topologies on $\T$ and we will need this generality later when we introduce the compactly generated Grothendieck topology. For the remainder of this subsection, let $J$ be an arbitrary subcanonical Grothendieck topology on $\T$.

\begin{dfn}
Given a topological groupoid $\G$, a (left) \textbf{$\G$-space} is a space $E$ equipped with a \textbf{moment map} $\mu:E \to \G_0$ and an \textbf{action map} $\rho:\G_1 \times_{\G_0} E \to E$,
where
$$\xymatrix{\G_1 \times_{\G_0} E \ar[r]  \ar[d] & E \ar^-{\mu}[d] \\
\G_1 \ar^-{s}[r] & \G_0\\}$$
is the fibered product, such that the following conditions hold:

\begin{itemize}
\item[i)] $\left(gh\right) \cdot e = g \cdot \left(h \cdot e\right)$ whenever $e$ is an element of $E$ and $g$ and $h$ elements of $\G_1$ with domains such that the composition makes sense
\item[ii)] $1_{\mu\left(e\right)} \cdot e =e$ for all $e \in E$
\item[iii)] $\mu\left(g \cdot e\right) = t \left(g\right)$ for all $g \in \G_1$ and $e \in E$.
\end{itemize}
\end{dfn}

\begin{dfn}
Suppose that $\G \acts E$ is a $\G$-space. Then the \textbf{action groupoid} $\G \ltimes E$ is defined to be the topological groupoid whose arrow space is $\G \times_{\G_0} E$ and whose object space is $E$. An element $$\left(g,e\right) \in \G \times_{\G_0} E \subset \G \times E$$ is viewed as an arrow from $e$ to $g \cdot e$. Composition is defined by the rule

$$\left(g,e\right) \cdot \left(h,g\cdot e\right)= \left(hg,e\right).$$
\end{dfn}

\begin{dfn}\label{dfn:tran}
Given a topological groupoid $\G$, the \textbf{translation groupoid} $\Ee \G$ is defined to be the action groupoid $\G \ltimes \G_1$ with respect to the action of $\G$ on $\G_1$ by multiplication.
\end{dfn}

\begin{dfn}
A (left) \textbf{$\G$-bundle} over a space $X$  (with respect to $J$) is a (left) $\G$-space $P$ equipped with $\G$-invariant \textbf{projection map} $$\pi:P \to X$$ which \textbf{admits local sections} with respect to the Grothendieck topology $J$. This last condition means that there exists a covering family $\mathcal{U}=\left(U_i \to X\right)_i$ in $J$ and morphisms $\sigma_i:U_i \to P$ called \textbf{local sections} such that the following diagram commutes for all $i$:

$$\xymatrix{& P \ar[rd]^-{f} &\\
U_i \ar[ru]^-{\sigma_i}  \ar[rr] & & X.}$$
\label{dfn:locs}

This condition is equivalent to requiring that the projection map is a $J$-cover.

Such a $\G$-bundle is called ($J$)-\textbf{principal} if the induced map, $$\G_1 \times_{\G_0} P \to P \times_{X} P$$ is a homeomorphism.
\end{dfn}

We typically denote such a principal bundle as

$$\xymatrix @R=2pc @C=0.15pc {\G_1  \ar@<+.7ex>[d] \ar@<-.7ex>[d] & \acts & P \ar_{\mu}[lld] \ar^{\pi}[d] \\
\G_0 && X.}$$

\begin{rmk}
To ease terminology, for the rest of this section, the term principal bundle, will refer to a $J$-principal bundle for our fixed topology $J$.
\end{rmk}

\begin{dfn}\label{dfn:unitbundle}
Any topological groupoid $\G$ determines a principal $\G$-bundle over $\G_0$ by

$$\xymatrix @R=2pc @C=0.15pc {\G_1  \ar@<+.7ex>[d] \ar@<-.7ex>[d] & \acts & \G_1 \ar_{t}[lld] \ar^{s}[d] \\
\G_0 && \G_0}$$
called the \textbf{unit bundle}, $1_\G$.
\end{dfn}

\begin{dfn}
Given $P$ and $P'$, two principal $\G$-bundle over $X$, a continuous map $f:P \to P'$ is a \textbf{map of principal bundles} if it respects the projection maps and is $\G$-equivariant. It is easy to check that any such map must be an isomorphism of principal bundles.
\end{dfn}

\begin{dfn}
Let $f:Y \to X$ be a map and suppose that $P$ is a principal $\G$-bundle over $X$. Then we can give $Y\times_{X} P \to Y$ the structure of a principal $\G$-bundle $f^*\left(P\right)$ over $Y$, called the \textbf{pull-back bundle}, in the obvious way.
\end{dfn}

\begin{dfn}\label{dfn:gaugegpd}
If $\G$ is topological groupoid and $P$ is a principal $\G$-bundle over $X$, then we define the \textbf{gauge groupoid} $Gauge\left(P\right)$ to be the following topological groupoid:

The fibered product

$$\xymatrix{P \times_{\G_0} P \ar[r] \ar[d] & P \ar[d]^-{\mu}\\
P \ar[r]^-{\mu} & \G_0}$$
admits a left-$\G$-action with moment map $\tilde \mu\left(\left(p,q\right)\right)=\mu\left(p\right)$ via

$$g \cdot \left(p,q\right)=\left(g\cdot p,g \cdot q\right).$$
The arrow space of $Gauge\left(P\right)$ is the quotient ${\raise0.7ex\hbox{$P \times_{\G_0} P$} \!\mathord{\left/
 {\vphantom {P \times_{\G_0} P \G}}\right.\kern-\nulldelimiterspace}
\!\lower0.7ex\hbox{$\G$}}$ and the object space is $X$. An equivalence class $\left[\left(p,q\right)\right]$ is viewed as an arrow from $\pi\left(p\right)$ to $\pi\left(q\right)$ (which is well defined as $\pi$ is $\G$-invariant.) Composition is determined by the rule

$$\left[\left(p,q\right)\right] \cdot \left[\left(q',r\right)\right]=\left[\left(p,g \cdot r\right)\right]$$
where $g$ is the unique element of $\G_1$ such that $g \cdot q'=q$.
\end{dfn}

\begin{dfn}\label{dfn:bibundle}
Let $\G$ and $\h$ be topological groupoids. A (left) \textbf{principal $\G$-bundle over $\h$} is a (left) principal $\G$-bundle

$$\xymatrix @R=2pc @C=0.15pc {\G_1  \ar@<+.7ex>[d] \ar@<-.7ex>[d] & \acts & P \ar_{\mu}[lld] \ar^{\nu}[d] \\
\G_0 && \h_0}$$
over $\h_0$, such that $P$ also has the structure of a right $\h$-bundle with moment map $\nu$, with the $\G$ and $\h$ actions commuting in the obvious sense. We typically denote such a bundle by

$$\xymatrix @R=2pc @C=0.15pc {\G_1  \ar@<+.7ex>[d] \ar@<-.7ex>[d] & \acts & P \ar_{\mu}[lld] \ar^{\nu}[rrd] & \acted & \h_1 \ar@<+.7ex>[d] \ar@<-.7ex>[d]\\
\G_0 & & & &\h_0.}$$
\end{dfn}

\begin{rmk}
To a continuous functor $\varphi:\h  \to \G$, one can canonically associate a principal $\G$-bundle over $\h$. It is defined by putting the obvious (right) $\h$-bundle structure on the total space of the pullback bundle $\varphi_0^*\left(1_\G\right)=\G_1 \times_{\G_0} \h_0$.
\end{rmk}

\subsection{Bicategories of Topological Groupoids}
\begin{dfn}\label{dfn:JMor}
A continuous functor $\varphi:\h \to \G$ between two topological groupoids is a $J$-\textbf{Morita equivalence} if the following two properties hold:

\begin{itemize}
\item[i)] (Essentially Surjective)

The map $t \circ pr_1:\G_1 \times_{\G_0} \h_0 \to \G_0$ admits local sections with respect to the topology $J$, where $\G_1 \times_{G_0} \h_0$ is the fibered product

$$\xymatrix{\G_1 \times_{G_0} \h_0 \ar[r]^-{pr_2} \ar[d]_-{pr_1} & \h_0 \ar[d]^-{\varphi} \\
\G_1 \ar[r]^-{s} & \G_0.}$$
\item[i)] (Fully Faithful)

The following is a fibered product

$$\xymatrix{\h_1 \ar[r]^-{\varphi} \ar[d]_-{\left\langle {s,t} \right\rangle} &\G_1 \ar[d]^-{\left\langle {s,t} \right\rangle}  \\
\h_0 \times \h_0 \ar[r]^-{\varphi \times \varphi} & \G_0 \times \G_0.}$$
\end{itemize}
\end{dfn}

\begin{rmk}
If $\mathcal{U}$ is a $J$-cover of the object space $\G_0$ of a topological groupoid $\G$, then the induced map $\G_{\mathcal{U}} \to \G$ is a $J$-Morita equivalence.
\end{rmk}

\begin{rmk}
We will again suppress the reference to the Grothendieck topology $J$; for the rest of the section a Morita equivalence will implicitly mean a $J$-Morita equivalence for our fixed topology $J$. A Morita equivalence with respect to the open cover topology will be called an ordinary Mortia equivalence.
\end{rmk}

\begin{rmk}
The property of being a Morita equivalence is weaker than being an equivalence in the bicategory $\T Gpd$. In fact, a Morita equivalence is an equivalence in $\T Gpd$ if and only if $t \circ pr_1$ admits a global section. However, any Morita equivalence does induce an equivalence in the bicategory $Gpd$ after applying the forgetful $2$-functor. Morita equivalences are sometimes referred to as weak equivalences.
\end{rmk}

We denote by $W_J$ the class of Morita equivalences. The class $W_J$ admits a right calculus of fractions in the sense of \cite{Dorette}. There is a bicategory $\T Gpd\left[W_J^{-1}\right]$ obtained from the bicategory $\T Gpd$ by formally inverting the Morita equivalences. A $1$-morphism from $\h$ to $\G$ in this bicategory is a diagram of continuous functors

$$\xymatrix{&\K \ar[ld]_{w} \ar[rd]&\\
\h & & \G}$$
such that $w$ is a Morita equivalence.
\begin{dfn}\label{dfn:genhom}
Such a diagram is called a \textbf{generalized homomorphism}.
\end{dfn}
There is also a well defined notion of a $2$-morphism. For details see \cite{Dorette}.

There is in fact another bicategory which one can construct from the bicategory $\T Gpd$. This bicategory, denoted by $Bun^J \T Gpd$ again has the same objects as $\T Gpd$. A $1$-morphism between two topological groupoids $\h$ and $\G$ is a left principal $\G$-bundle over $\h$ (with respect to $J$). There is a well defined way of composing principal bundles, for details see \cite{Poisson} or \cite{HSMaps}. A $2$-morphism between two such principal bundles is a biequivariant map (a map which is equivariant with respect to the left $\G$-action and the right $\h$-action). A principal $\G$-bundle over $\h$ is an equivalence in this bicategory if and only if the underlying $\h$-bundle is also principal.

By the remark following definition \ref{dfn:bibundle}, there is a canonical inclusion $$\T Gpd \hookrightarrow Bun^J \T Gpd$$ which sends Morita equivalences to equivalences. Therefore, there is an induced  map $$\T Gpd\left[W_J^{-1}\right] \to Bun^J \T Gpd$$ of bicategories.

\begin{thm}\label{thm:bunmor}
The induced map $$\T Gpd\left[W_J^{-1}\right] \to Bun^J \T Gpd$$ is an equivalence of bicategories.
\end{thm}

This theorem is well known. A $1$-categorical version of this theorem is proven in \cite{Poisson}, Section 2.6. The general result follows easily from \cite{Dorette}, Section 3.4.

\subsection{Topological Stacks}\label{app:topst}

Consider the bicategory $Gpd^{\T^{op}}$\ of weak presheafs over $\T$ \footnote{Technically, this is not well defined due to set-theoretic issues, however, this can be overcame by careful use of Grothendieck universes. We will not dwell on this and all such similar size issues in this paper.}, that is contravariant (possibly weak) $2$-functors from the category $\T$ into the $2$-categeory of (essentially small) groupoids $Gpd$. Let $G$ be a topological group. A standard example would be the weak presheaf that assigns to each space the category of principal $G$-bundles over that space (this category is a groupoid). More generally, let $\G$ be a topological groupoid. Then $\G$ determines a weak presheaf on $\T$ by the rule
\begin{equation*}
X \mapsto \Hom_{\T Gpd}\left(\left(X\right)^{(id)},\G\right).
\end{equation*}

This defines an extended Yoneda $2$-functor $\tilde y:\T Gpd \to Gpd^{\T^{op}}$ and we have the obvious commutative diagram

$$\xymatrix{\T  \ar[d]_{\left(\mspace{2mu} \cdot \mspace{2mu}\right)^{(id)}} \ar[r]^{y} & \Set^{\T^{op}} \ar^{\left(\mspace{2mu} \cdot \mspace{2mu}\right)^{(id)}}[d]\\
\T Gpd \ar_{\tilde y}[r] & Gpd^{\T^{op}}.}$$

\begin{rmk}
$\tilde y$ preserves all weak limits.
\label{lem:ylim}
\end{rmk}

Given a subcanonical Grothendieck topology $J$ on $\T$, we denote by $\left[\G\right]_J$ the associated stack on $\left(\T,J\right)$, $a_J \circ \tilde y\left(\G\right)$, where $a_J$ is the stackification $2$ functor (Definition \ref{dfn:stackification}).

It can be checked that

\begin{equation*}
\left[\G\right]_J \left(X\right) \simeq \underset{\mathcal{U} \in Cov\left(X\right)} \hc \Hom_{\T Gpd} \left(X_{\mathcal{U}},\G\right),
\end{equation*}
where the weak $2$-colimit above is taken over a suitable bicategory of $J$-covers. For details in the case $J$ is the open cover topology, see \cite{Andre}.

\begin{rmk}
Since $a_J$ preserves finite weak limits, it follows that $\left[\mspace{2mu} \cdot \mspace{2mu}\right]_J$ does as well.
\end{rmk}

\begin{dfn}
A stack $\X$ on $\left(\T,J\right)$ is \textbf{presentable} if it is equivalent to $\left[\G\right]_J$ for some topological groupoid $\G$. In this case, $\G$ is said to be a \textbf{presentation} of $\X$.

We denote the full sub-bicategory of $St_J\left(\T\right)$ consisting of presentable stacks by $\Prst$.
\label{dfn:presentable}
\end{dfn}

\begin{dfn}\label{dfn:tst}
A \textbf{topological stack} is a presentable stack for the open cover topology on $\T$. We shall denote the topological stack associated to a topological groupoid $\G$ by $\left[\G\right]$. We shall also denote the bicategory of topological stacks by $\Tst$.
\end{dfn}

\begin{thm}
The $2$-functor

$$a_{J} \circ \tilde y:\T Gpd \to \Prst$$
induces an equivalence of bicategories

$$P_J:\T Gpd \left[W_J^{-1}\right] \stackrel{\sim}{\longrightarrow} \Prst.$$
\label{thm:bicatp1}
\end{thm}

This theorem is well known. For example, see \cite{Dorette} for the case of \'etale topological groupoids and topological stacks with an \'etale atlas, and also for \'etale Lie groupoids and \'etale differentiable stacks. The preprint \cite{diffg} contains much of the necessary ingredients for the proof of the case of general Lie groupoids and differentiable stacks in its so called Dictionary Lemmas. Similar statements in the case of algebraic stacks can be found in \cite{Naum}. The general theorem follows again from an easy application of \cite{Dorette}, Section 3.4.

Hence all three bicategories, $\Prst$, $\T Gpd \left[W_J^{-1}\right]$ and $Bun^J\T Gpd$ are equivalent.

\begin{cor}
If $\G$ is a topological groupoid, then the associated presentable stack $\left[\G\right]_J$ (is equivalent to one that) assigns to each topological space $T$, the groupoid of principal $\G$-bundles over $T$, $Bun^J_{\G}\left(T\right)$.
\end{cor}

\begin{dfn}\cite{Metzler}\label{dfn:covmor}
A morphism $\varphi:\X \to \Y$ between weak presheaves is said to be a $J$-\textbf{covering morphism} if for every space $X$ and every object $x \in \Y\left(X\right)_0$, there exists a $J$-covering family $\mathcal{U}=\left(f_i:U_i \to X\right)_i$ of $X$ such that for each $i$ there exists an object $x_i \in \X(U_i)_0$ and a(n) (iso)morphism $$\alpha_i:\varphi\left(\X\left(f_i\right)\left(x_i\right)\right) \to \Y\left(f_i\right)\left(x\right).$$

If $\X$ and $\Y$ are both $J$-stacks, $J$-covering morphisms are referred to as $J$-\textbf{epimorphisms}\cite{Metzler},\cite{NoohiF}.
\end{dfn}

There is a more intrinsic description of presentable stacks:

\begin{dfn}
A $J$-\textbf{atlas} for a stack $\X$ over $\left(\T,J\right)$ is a \emph{representable} (see Definition \ref{dfn:repr}) $J$-epimorphism $p:X \to \X$ from a topological space $X$.
\end{dfn}

\begin{prop}\cite{NoohiF}
A stack $\X$ is presentable if and only if it has an atlas.
\label{prop:atlas}
\end{prop}

Suppose that $p:X \to \X$ is an atlas. Consider the weak fibered product

$$\xymatrix{X \times_{\X} X \ar[d] \ar[r] & X \ar[d]^-{p}\\
X \ar[r]^-{p} & \X.}$$

Then $X \times_{\X} X \rightrightarrows X$ has the structure of a topological groupoid whose associated stack is $\X$.

Conversely, given a topological groupoid $\G$, the canonical map of groupoids $\left(\G_0\right)^{id} \to \G$ (which is the identity objects and $u$ on arrows)
produces a morphism $p:\G_0 \to \left[\G\right]_{J}$ which is a representable epimorphism.

We end this section by stating a technical lemma.

\begin{lem}\label{lem:319}
The presentable stack $\left[\G\right]_{J}$ is the weak colimit of the following diagram of representables:
$$\h_2 \rrrarrow \h_1 \rrarrow \h_0.$$
where the maps are the face maps of the truncated simplicial diagram.
\end{lem}

\begin{proof}
This is the content of \cite{NoohiF} Proposition 3.19.
\end{proof}

\end{alphasection}

\bibliography{Main}
\bibliographystyle{hplain}
\end{document}